\documentclass[reqno]{amsart}
\usepackage[left=3cm,right=3cm,top=2cm,bottom=2cm,includeheadfoot]{geometry}
\usepackage[shortlabels]{enumitem}
\usepackage{amsmath,amssymb,latexsym,esint,cite,mathrsfs,amscd,amsxtra}
\usepackage{amsfonts}
\usepackage{amsmath}
\usepackage{amssymb}
\usepackage{latexsym}
\usepackage{amsthm}
\usepackage{hyperref}
\usepackage{mathtools}
\hypersetup{
	colorlinks=true,
	linkcolor=blue,
	filecolor=magenta,
	urlcolor=cyan,
}
\usepackage{array}
\usepackage{color}
\usepackage{tcolorbox}
\hypersetup{linkcolor=blue, colorlinks=true ,citecolor = red}
\DeclareMathOperator*{\essinf}{ess\,inf}
\DeclareMathOperator*{\esssup}{ess\,sup}

\DeclareMathOperator*{\supp}{supp}
\newcolumntype{M}[1]{>{\centering\arraybackslash}m{#1}}
\newcolumntype{N}{@{}m{0pt}@{}}
\newcommand*\diff{\mathop{}\!\mathrm{d}}

\def\4#1{\mathbb{#1}}
\def\5#1{\mathcal{#1}}
\def\R{{\4R}}
\def\N {{\4N}}

\newcommand{\eps}{\varepsilon}

\newcommand{\la}{{\lambda}}

\makeatletter
\@namedef{subjclassname@2020}{\textup{2020} Mathematics Subject Classification}
\makeatother

\begin{document}
	\title[Critical double phase anisotropic problems]
	{Double phase anisotropic variational problems involving critical growth}
	\author[K. Ho]{Ky Ho}
	\address[K. Ho]{Institute of Applied Mathematics, University of Economics Ho Chi Minh City (UEH), 59C, Nguyen Dinh Chieu Street, District 3, Ho Chi Minh City, Vietnam}
	\email{kyhn@ueh.edu.vn}
	
			\author[Y.-H. Kim]{Yun-Ho Kim$^{*}$}
	\address[Y.-H. Kim]{Department of Mathematics Education, Sangmyung University, Seoul, 03016, Korea}
	\email{kyh1213@smu.ac.kr}
	
	\author[C. Zhang]{Chao Zhang}
	\address[C. Zhang]{School of Mathematics and Institute for Advanced Study in Mathematics, Harbin Institute of Technology, Harbin 150001, PR China}
	\email{czhangmath@hit.edu.cn}

\thanks{* Corresponding author.}
	
	\date{}
	\subjclass[2020]{35J20, 35J60, 35J70, 47J10, 46E35.}

	\keywords{Variable exponent elliptic operator, variable exponent Orlicz--Sobolev spaces, critical growth, concentration-compactness principle, variational methods}

	\begin{abstract}
		In this paper, we investigate some existence results for double phase anisotropic variational problems involving critical growth. We first establish a Lions type concentration-compactness principle and its variant at infinity for the solution space, which are our independent interests. By employing these results, we obtain a nontrivial nonnegative solution to problems of generalized concave-convex type. We also obtain infinitely many solutions when the nonlinear term is symmetric. Our results are new even for the $p(\cdot)$-Laplace equations.
	\end{abstract}
	\maketitle
	\numberwithin{equation}{section}
	\newtheorem{theorem}{Theorem}[section]
	\newtheorem{lemma}[theorem]{Lemma}
	\newtheorem{proposition}[theorem]{Proposition}
	\newtheorem{corollary}[theorem]{Corollary}
	\newtheorem{dn}[theorem]{Definition}
	\newtheorem{example}[theorem]{Example}
	\newtheorem{remark}[theorem]{Remark}
	\newtheorem{assumption}[theorem]{Assumption}
	
	\allowdisplaybreaks

	\newcommand{\abs}[1]{\left\lvert#1\right\rvert}
	\newcommand{\norm}[1]{|\!|#1|\!|}
	\newcommand{\Norm}[1]{\mathinner{\Big|\!\Big|#1\Big|\!\Big|}}
	\newcommand{\curly}[1]{\left\{#1\right\}}
	\newcommand{\Curly}[1]{\mathinner{\mathopen\{#1\mathclose\}}}
	\newcommand{\round}[1]{\left(#1\right)}
	\newcommand{\bracket}[1]{\left[#1\right]}
	\newcommand{\scal}[1]{\left\langle#1\right\rangle}
	\newcommand{\Div}{\text{\upshape div}}
	\newcommand{\dotsr}{\dotsm}
	\newcommand{\ra}{\rightarrow}
	\newcommand{\ran}{\rangle}
	\newcommand{\lan}{\langle}
	\newcommand{\ol}{\overline}
	\newcommand{\e}{{\varepsilon}}
	\newcommand{\al}{{\alpha}}
	\newcommand{\tid}[1]{{\widetilde{#1}}}
	\newcommand{\M}{\mathcal{M}}
	\newcommand{\W}{\mathcal{W}}

	%########################## SECTION I. INTRODUCTION #############################%
	\section{Introduction}
	In this paper, we investigate the existence of a nontrivial nonnegative solution and infinitely many solutions to the following double phase anisotropic problem of Schr\"{o}dinger--Kirchhoff type:
\begin{equation}\label{Prob}
	- M\left(\int_{\mathbb{R}^N} A(x,\nabla u)\diff x\right)\operatorname{div}a(x,\nabla u)+V(x)|u|^{\alpha(x)-2}u=\lambda f(x,u)+b(x)|u|^{t(x)-2}u \quad \text{in  }   \4{R}^N,
\end{equation}
where  $a(x,\xi)=\nabla_\xi A(x,\xi)$ behaves like $|\xi|^{q(x)-2}\xi$ for small $|\xi|$ and like $|\xi|^{p(x)-2}\xi$ for large $|\xi|$ with $1< p(\cdot)<q(\cdot)<N$; $t\in C(\R^N)$ satisfies $1<\alpha(\cdot)\leq t(\cdot)\leq \frac{Np(\cdot)}{N-p(\cdot)}=:p^\ast(\cdot)$ with
\begin{equation}\label{critical_C}
\mathcal{C}:=\left\{x\in\mathbb{R}^N:\, t(x)=p^\ast(x)\right\}\ne\emptyset;
\end{equation}
 $M: \, [0, \infty) \to \mathbb{R}$ is a real function, which is continuous and nondecreasing on some interval $[0,\tau_0)$; $V: \mathbb{R}^N  \to
 \mathbb R$ is measurable and positive a.e. in $\R^N$; $f : \mathbb{R}^N \times \mathbb R \to
\mathbb R$ is a Carath\'{e}odory function and is of subcritical growth; $b\in L^\infty(\mathbb{R}^N)$ and $b(x)>0$ for a.e. $x\in\R^N$; and $\lambda>0$ is a parameter.

\vskip5pt
Throughout this paper, for $m,n\in C(\mathbb{R}^N),$ by $m(\cdot)\ll n(\cdot)$ we mean $\inf_{x\in\mathbb{R}^N} (n(x)-m(x))>0$. Moreover, we denote
$$C^{0,1}_+(\4{R}^N):=\left\{r\in C(\R^N): \, r \text{ is Lipschitz continuous and } 1\ll r(\cdot)\right\}.$$
We make the following assumptions:
\begin{itemize}
	\item [$(A1)$] The function $A \in C\left(\4{R}^N \times \mathbb R^N \right)$ is continuously differentiable with respect to the second variable with $\nabla_\xi A=:a$, and verifies $A(x,-\xi)=A(x,\xi)$ and $A(x,0)=0$ for all $(x,\xi)\in \4{R}^N \times \mathbb R^N$.
	\item[$(A2)$] $A(x,\cdot)$ is strictly convex in $\4{R}^N$ for all $x\in \4{R}^N$.
	\item[$(A3)$] There exist positive constants $N_1,N_2$ and variable functions $p$ and $q$ such that for all $(x,\xi) \in \4{R}^N \times \4{R}^N$, it holds:
	\[a(x, \xi)\cdot \xi  \geq \begin{cases}
		N_1|\xi|^{p(x)} & \text{ if } \,\, |\xi| \gg 1\\
		N_1|\xi|^{q(x)} & \text{ if } \,\, |\xi| \ll 1
	\end{cases} \quad \text{and} \quad |a(x,\xi)| \leq \begin{cases}
		N_2|\xi|^{p(x)-1} & \text{ if } \,\, |\xi| \gg 1\\
		N_2|\xi|^{q(x)-1} & \text{ if } \,\, |\xi| \ll 1
	\end{cases}. \]
	
	\item[$(A4)$] $p,q \in C^{0,1}_+(\4{R}^N)$ and $p(\cdot) \ll q(\cdot) \ll \min\{N,p^\ast(\cdot)\}$.
	\item [$(A5)$] $a(x,\xi)\cdot \xi \leq s(x) A(x,\xi)$ for any $(x,\xi) \in \4{R}^N \times \4{R}^N$, where $s \in C^{0,1}_+(\4{R}^N)$  satisfies $q(\cdot) \leq s(\cdot) \ll p^\ast (\cdot)$.
	\item [$(A6)$] $A$ is uniformly convex, that is, for any $\eps\in (0,1)$, there exists $\delta(\eps)\in (0,1)$ such that $|u-v|\leq \eps \max\{|u|,|v|\}$ or $A(x,\frac{u+v}{2})\leq \frac{1}{2}(A(x,u)+A(x,v))$ for any $x,u,v\in\R^N$.
\end{itemize}

Note that from $(A1)$ we easily obtain
\begin{equation}\label{a-odd}
	a(x,-\xi)=-a(x,\xi), \ \ \forall (x,\xi)\in \R^N\times\R^N.
\end{equation}
Moreover, as shown in \cite[Eqn. (3)]{Zhang-Radulescu.2018} there exist positive constants $M_1, M_2$ such that
\begin{equation}\label{Est1}
	\begin{rcases}
		M_1 |\xi|^{p(x)}, & |\xi| > 1\\
		M_1 |\xi|^{q(x)}, & |\xi| \leq 1
	\end{rcases} \leq A(x,\xi) \leq a(x,\xi)\cdot \xi \leq  \begin{cases}
		M_2|\xi|^{p(x)}, & |\xi| >1\\
		M_2|\xi|^{q(x)}, & |\xi| \leq 1
	\end{cases},  \ \ \forall (x,\xi) \in \4{R}^N \times \4{R}^N.
\end{equation}

The basic assumptions for the potential $V$ and the variable exponent $\alpha$ to get the desired solution space are the following:
\begin{itemize}
	\item  [$(V1)$] $V \in L^1_{\textup{loc}}(\4{R}^N)$ and $\essinf_{x\in\4{R}^N}V(x)=V_0>0$.
	\item[$(P1)$ ] $\alpha\in C(\R^N)$ such that $1 \ll \alpha(\cdot) \ll p^\ast(\cdot)\frac{N-1}{N}$ and $\alpha(\cdot) \leq p^\ast(\cdot)\frac{q'(\cdot)}{p'(\cdot)}$ in $\R^N$ with  $r'(\cdot):=\frac{r(\cdot)}{r(\cdot)-1}$ for $r\in C(\R^N)$ satisfying $1 \ll r(\cdot)$.
\end{itemize}
For obtaining our concentration-compactness principle, we will need the following further assumption:
\begin{itemize}
	\item [$(V2)$] $\esssup_{x\in B_R}\, V(x)<\infty$ for any $R>0$.
\end{itemize}
Moreover, for covering some classes of subcritical terms we will include the following assumption on $V$ in some situations:
\begin{itemize}
	\item [$(V3)$]  $V(x)\to\infty$ as $|x|\to\infty$.
\end{itemize}	

\smallskip
The studies of differential equations and variational problems with nonhomogeneous operators and non-standard growth conditions have attracted extensive attentions during the last decades. The interest in the equations associated by nonhomogeneous nonlinearities has consistently developed in light of the pure or applied mathematical perspective to illustrate some concrete phenomena arising from nonlinear elasticity, plasticity theory, and plasma physics. Let us recall some related results by way of motivation. Azzollini et al. in \cite{Azzollini2014, Azzollini2015} introduced a new class of nonhomogeneous operators with a variational structure:
$$
-\operatorname{div}(\phi'(|\nabla u|^{2})\nabla u),
$$
where $\phi\in C^1(\mathbb{R}^+,\mathbb{R}^+)$ has a different growth near zero and infinity. Such a behaviour occurs if
$$
\phi(t)=2((1+t)^{\frac{1}{2}}-1),
$$
which corresponds to the prescribed mean curvature operator defined by
$$
\operatorname{div}\left(\frac{\nabla u}{\sqrt{1+|\nabla u|^2}}\right).
$$
In particular, Azzollini, d'Avenia and Pomponio in \cite{Azzollini2014} proved the existence of a nontrivial nonnegative radially symmetric solution for the quasilinear elliptic problem
\begin{align}\label{Prob3}
	\begin{cases} \displaystyle
		-\operatorname{div} (\phi'(|\nabla u|^{2})\nabla u)+|u|^{\alpha-2}u=|u|^{s-2}u
		\quad \textmd{in }  \mathbb{R}^N,\\
		u\to 0, \quad\textmd{as }  |x|\rightarrow\infty,
	\end{cases}
\end{align}
where $N\geq2$, $\phi(t)$ behaves like $t^{q/2}$ for small $t$ and $t^{p/2}$ for large $t$, and
\begin{equation}\label{4}
	1<p<q<N, \quad 1<\alpha\leq\frac{p^*q'}{p'}, \quad \max\{\alpha,q\}<s<p^*
\end{equation}
with $p'=p/(p-1)$ and $q'=q/(q-1)$. Under the above assumption (\ref{4}), Chorfi and  R\u{a}dulescu \cite{Chorfi2016} considered the standing wave solutions for the following Schr\"{o}dinger equation with unbounded potential:
$$
-\operatorname{div}(\phi'(|\nabla u|^{2})\nabla u)+a(x)|u|^{\alpha-2}u=f(x,u) \quad \textmd{in }  \mathbb{R}^N,
$$
where the nonlinearity $f:\mathbb{R}^N\rightarrow\mathbb{R}$ also satisfies the subcritical growth and $a:\Bbb R^N \to (0,\infty)$ is a singular potential satisfying some conditions.

Anisotropic partial differential equations have recently gained significant attention, thanks to their applications in double and multiphase variational energies, along with their relevance in integral form anisotropic energies. For insights into this topic, we refer the reader to the survey paper \cite{M2020} and the references therein. Problem \eqref{Prob3} is relevant to double phase anisotropic phenomenon, in the sense that the differential operator has a different growth near zero and infinity. The double phase problems are described by the following functional
$$\int_\Omega H(x,\nabla u)\,dx$$
with the so-called $(p,q)$-growth conditions:
$$
	c|\xi|^p\leq H(x,\xi)\leq C(|\xi|^q+1).
$$
This $(p,q)$-growth condition was first treated by Marcellini \cite{M1989,M1991,M1993,M1996} and it has been extensively studied in the last decades. Particularly double phase functionals have been introduced by Zhikov in the context of homogenization and
Lavrentiev's phenomenon \cite{Z1986,Z1995}. After that it engages the enormous researchers’ attention to the development of both theoretical and applications aspects of various double phase differential problems. For an overview of the subject, we refer the readers to the survey paper \cite{Min-Rad2021}. Regularity theory for double phase functionals had been an unsolved issue for a while. However, we would like to mention  a series of remarkable papers by Mingione et al. \cite{BCM,BCM2,BCM3,CM,CM2,CM3}. Also we refer to the works of Bahrouni-R\u{a}dulescu-Repov\v s~\cite{BRR}, Byun-Oh \cite{Byun-Oh-2020}, Colasuonno-Squassina \cite{CS1},
Gasi\'nski-Winkert~\cite{GW20b,GW21}, Liu-Dai \cite{LD}, Papageorgiou-R\u{a}dulescu-Repov\v s~\cite{PRR,PRR20}, Perera-Squassina \cite{PS}, Ragusa-Tachikawa \cite{RT20}, Zhang-R\u{a}dulescu~\cite{Zhang-Radulescu.2018},  Zeng-Bai-Gasi\'nski-Winkert~\cite{Zeng-Bai-Gasinski-Winkert-CVPDEs-2020,Zeng-Bai-Gasinski-Winkert-ANONA-2021}.

In recent years the study on partial differential equations with variable exponent has received an increasing deal of attention because they can be perceived as their application in the mathematical modeling of many physical phenomena occurring in diverse studies related to electro-rheological fluids, image processing and the flow in porous media {\it etc}. There are many reference papers associated with the study of elliptic problems with variable exponent; see \cite{Cruz,Diening,Rad-Rep} and the references therein for more background on applications. Also we refer the readers to \cite{CGHW,KKOZ22,RT20,Zhang-Radulescu.2018,SRRZ} for double phase differential problems with variable exponent. In this directions, Zhang and R\u{a}dulescu in a recent work \cite{Zhang-Radulescu.2018} extended the results in \cite{Chorfi2016} to the more general variable exponent case
\begin{equation}\label{EQ-1}
-{\rm div}\, {\Phi}(x,\nabla u)+V(x)|u|^{\alpha(x)-2}u=f(x,u)\quad \textmd{in }  \mathbb{R}^N,
\end{equation}
where the differential operator ${\Phi}(x,\xi)$  has behaviors like $|\xi|^{q(x)-2}\xi$ for small $\xi$ and like $|\xi|^{p(x)-2}\xi$ for large $\xi$, $1<\alpha(\cdot)\le p(\cdot)<q(\cdot)<N$. In order to analyze problem \eqref{EQ-1}, the authors gave useful elementary properties of a function space, called the variable exponent Orlicz--Sobolev space that is a generalization of Orlicz--Sobolev space setting in \cite{Azzollini2014}. In particular, they provided fundamental imbedding results in the solution space, such as the Sobolev imbedding and the compact imbedding when the potential $V$ satisfies conditions $(V1)$ and $(V3)$. Also they obtained some topological properties for the energy functional corresponding to \eqref{EQ-1}. Motivated by this work, the authors in \cite{SRRZ} discussed the existence of multiple solutions to problem \eqref{EQ-1} with $V\equiv 1$ when nonlinear term $f$ has concave–convex nonlinearities. Very recently, Cen et al. in \cite{KCKS22} discussed the existence of multiple solutions to double phase anisotropic variational problems for the case of a combined effect of concave–convex nonlinearities:
\begin{equation*}
-\text{div}\, {\Phi}(x,\nabla u)+V(x)\left|u\right|^{\alpha(x)-2}u=\lambda a(x)|u|^{r(x)-2}u+f(x,u)\quad \text{in }\R^N,
\end{equation*}
where $\lambda >0$ is a parameter, $r\in C(\R^{N})$ satisfies $1<r(\cdot)<p(\cdot)$, $a$ is an appropriate potential function defined in $(0,\infty)$, and $f:\R^{N}\times \R^{N} \to \R$ is a Carath\'eodory function. Especially the superlinear (convex) term $f$ substantially fulfils a weaker condition as well as Ambrosetti--Rabinowitz condition.

The present paper can be seen as a continuation of the earlier works \cite{Zhang-Radulescu.2018, SRRZ,KCKS22} to the case of problem \eqref{Prob} with critical growth and containing a Kirchhoff term. The critical problem was originally studied in the pioneer paper by Brezis-Nirenberg \cite{Brezis} dealing with Laplace equations. Since then many researchers have been interested in such problems and there have been extensions of \cite{Brezis} in many directions. As we know, one of the difficulties in studying elliptic equations in an unbounded domain involving critical growth is the absence of compactness arising in relation with the variational approach. To overcome this difficulty, the concentration-compactness principles (the CCPs, for short), which were initially provided by Lions \cite{Lions1, Lions3}, and its variant at infinity \cite{BCS,BTW,Chabrowski} have been used. In particular, these principles have played a crucial role in showing the precompactness of a minimizing sequence or a Palais--Smale sequence. By making use of these CCPs or extending them to the suitable solution spaces, many authors have been successful to deal with critical problems involving elliptic equations of various types, see e.g., \cite{Bonder, Bon-Sai-Sil,Fu,HK.ANA2021,HKL2022,HKS,HS.NA2016} and references therein. Regarding the nonlocal Kirchhoff term, it was first introduced by Kirchhoff \cite{Kirchhoff-1876} to study an extension of the classical D'Alembert's wave equation for free vibrations of elastic strings. Elliptic problems of Kirchhoff type have a strong background in several applications in physics and have been intensively studied by many authors in recent years, for example \cite{APS,DAS,CH.2021,PXZ,FMPV} and the references therein.

The main feature of the present paper is to establish the existence and multiplicity of nontrivial nonnegative solutions to double phase anisotropic problem of Schr\"{o}dinger--Kirchhoff type with critical growth. To do this, we first establish a Lions type concentration-compactness principle and its variant at infinity for the solution space of \eqref{Prob}. As mentioned before, as compared with the $p(\cdot)$-Laplacian, the solution space belongs to the variable exponent Orlicz--Sobolev spaces and the general variable exponent elliptic operator $a$ has a different growth near zero and infinity. Hence this problem with nonhomogeneous operator has more complex nonlinearities than those of the previous works \cite{Fu,Fu-Zhang,HKL2022,HS.NA2016,HKS}, so more exquisite analysis has to be meticulously carried out when we obtain the CCPs for the solution space (see Theorems \ref{T.CCP} and \ref{T.CCP.infinity}). As far as we are aware, there were no such results for the variable exponent double phase anisotropic problems in this situation. This is one of novelties of this paper.

As an application of this result, we will be concerned with the existence of a nontrivial nonnegative solution and infinitely many solutions to a class of double phase problems involving critical growth. Concerning this, on a new class of nonlinearities which is a concave-convex type of nonlinearities we get the existence result of a nontrivial nonnegative solution via applying the Ekeland variational principle (see Theorem \ref{Theo.Sub1}). When the symmetry of the reaction term $f$ is additionally assumed, by employing the genus theory we derive that problem \eqref{Prob} admits infinitely solutions (see Theorem \ref{Theo.Sub}). Although the proofs of Theorems \ref{Theo.Sub1} and \ref{Theo.Sub} follow the basic idea of those in \cite{CH.2021}, we believe these consequences are new even in the constant exponent case as well as for the $p(\cdot)$-Laplace equations. This is another novelties of the present paper. To the best of our knowledge, this study is the first effort to develop some existence and multiplicity results to the variable exponent double phase anisotropic problems with critical growth because we observe a new class of nonlinearities which is of a generalized concave-convex type.

Our paper is organized as follows. In Section~\ref{Pre} we review some properties of the variable exponent spaces. In Section~\ref{CCP} we establish the concentration-compactness principles for variable exponent Orlicz--Sobolev space, which plays as the solution space of problem~\eqref{Prob}. The final main section, Section~\ref{Application}, is devoted to the study of the existence and multiplicity of nontrivial solutions to problem~\eqref{Prob}. In the appendix, we provide a proof for a helpful inequality presented in Section~\ref{Pre}.

%$$$$$$$$$$$$$$$$$$$$$$ SECTION 2. SUPERLINEAR CASE $$$$$$$$$$$$$$$$$$$$$$%
	\section{Preliminaries and Notations}\label{Pre}
\subsection{Lebesgue--Sobolev spaces with variable exponent} ${}$

\smallskip
In this subsection, we briefly recall the definition and some basic properties of the Lebesgue--Sobolev spaces with variable exponent, which were systematically studied in \cite{Diening}.

Let $\Omega$ be an open domain in $\4{R}^N$. Denote
$$\5{P}_+(\Omega):= \left\{ w:\Omega\to \R:  w \ \text{is Lebesgue measurable and positive a.e. in}\ \Omega\right \}$$
and
$$C_+(\Omega):= \left\{ m\in C(\overline{\Omega}):\ 1<m^-:=\inf_{x\in\Omega}\, m(x)\leq m^+:=\sup_{x\in\Omega}\, m(x)<\infty\right\}.$$
For $p\in C_+(\overline\Omega)$ and a $\sigma$-finite, complete measure $\mu$ in $\Omega,$ define the variable exponent Lebesgue space $L_\mu^{p(\cdot)}(\Omega)$ as
$$
L_\mu^{p(\cdot)}(\Omega) := \left \{ u : \Omega\to\mathbb{R}:\ u \ \hbox{ is}\  \mu-\text{measurable and }\ \int_\Omega |u(x)|^{p(x)} \;\diff\mu < \infty \right \}
$$
endowed with the Luxemburg norm
$$
|u|_{L_\mu^{p(\cdot)}(\Omega)}:=\inf\left\{\lambda >0:
\int_\Omega
\Big|\frac{u(x)}{\lambda}\Big|^{p(x)}\;\diff\mu\le1\right\},\quad \forall u\in L_\mu^{p(\cdot)}(\Omega).
$$
When $\mu$ is the Lebesgue measure, we write  $\diff x$, $L^{p(\cdot) }(\Omega) $  and $|\cdot|_{L^{p(\cdot)}(\Omega)}$  in place of $\diff \mu$, $L_\mu^{p(\cdot)}(\Omega)$ and $|\cdot|_{L_\mu^{p(\cdot)}(\Omega)}$, respectively. For $w\in \5{P}_+(\Omega)$, denote $L^{p(\cdot)}(w,\Omega):=L_\mu^{p(\cdot)}(\Omega)$ with $\diff \mu=w(x)\diff x$. We also denote
$$L_+^{p(\cdot) }(\Omega) :=\left\{  u\in L^{p(\cdot) }(\Omega):\, u \text{ is positive a.e. in } \Omega\right\}$$
and
$$p'(\cdot):=\frac{p(\cdot)}{p(\cdot)-1}.$$
Some basic properties of $L_\mu^{p(\cdot)}(\Omega)$ are listed in the next two propositions.

\begin{proposition} [\cite{Diening}]\label{norm-modular}
	 Define the modular $\rho:\ L_\mu^{p(\cdot)}(\Omega)\to\mathbb{R}$ as
	$$
	\rho(u):=\int_\Omega |u|^{p(x)}\;\diff \mu, \ \  \forall u\in
	L^{p(\cdot)}(\Omega).
	$$
	Then, for all $u\in L_\mu^{p(\cdot)}(\Omega)$  it holds that
	\begin{enumerate}[(i)]
		\item $\rho(u)>1$ $(=1;\ <1)$ if and only if \ $|u|_{L_\mu^{p(\cdot)}(\Omega)}>1$ $(=1;\ <1)$,
		respectively;
		\item if $|u|_{L_\mu^{p(\cdot)}(\Omega)}>1$, then $
		|u|_{L_\mu^{p(\cdot)}(\Omega)}^{p^-}\le \rho(u)\le
		|u|_{L_\mu^{p(\cdot)}(\Omega)}^{p^+}$;
		\item if $|u|_{L_\mu^{p(\cdot)}(\Omega)}<1$, then $
		|u|_{L_\mu^{p(\cdot)}(\Omega)}^{p^+}\le \rho(u)\le
		|u|_{L_\mu^{p(\cdot)}(\Omega)}^{p^-}$.
	\end{enumerate}
Consequently, we have
$$|u|_{L_\mu^{p(\cdot)}(\Omega)}^{p^{-}}-1\leq \rho (u)\leq |u|_{L_\mu^{p(\cdot)}(\Omega)}^{p^{+}}+1.$$
Furthermore, for $\{u_n\}_{n\in\N}\subset L_\mu^{p(\cdot)}(\Omega)$ it holds that
$$\lim_{n\to \infty }|u_n-u|_{L_\mu^{p(\cdot)}(\Omega)}=0\iff\lim_{n\to \infty }\rho (u_n-u)=0.$$
\end{proposition}
\begin{proposition}[\cite{FZ2001,Ko-Ra1991}] \label{Holder}
	The space $L^{p(\cdot)}(\Omega)$ is a separable and reflexive Banach space. 	Moreover, for any $u\in L^{p(\cdot)}(\Omega)$ and $v\in L^{p'(\cdot)}(\Omega)$ the following H\"older type inequality holds:
	\begin{equation*}
		\int_{\Omega}|uv|\diff x\leq 2|u|_{L^{p(\cdot)}(\Omega )}|v|_{L^{p'(\cdot)}(\Omega )}.
	\end{equation*}
\end{proposition}

\medskip
Define the Sobolev space $W^{1,p(\cdot)}(\Omega)$ as
\[W^{1,p(\cdot)}(\Omega):=\left\{u \in L^{p(\cdot)}(\Omega): |\nabla u| \in L^{p(\cdot)}(\Omega)\right\}\]
endowed with the norm:
\[\|u\|_{W^{1,p(\cdot)}(w,\Omega)}:=|u|_{L^{p(\cdot)}(\Omega)}+\big| |\nabla u|\big|_{L^{p(\cdot)}(\Omega)}.\]
We recall the following result obtained by X. Fan (see \cite[Theorem 1.1]{Fan2010}).
\begin{proposition}\label{P.Imb}
	Let $\Omega$ be a (bounded or unbounded) domain in $\4{R}^N$ satisfying the cone uniform condition. Suppose that $p \in C_+(\Omega)$ is Lipschitz in $\overline{\Omega}$ with $p^+<N$. Then there holds a continuous imbedding $W^{1,p(\cdot)}(\Omega) \hookrightarrow L^{q(\cdot)}(\Omega)$, for any $q\in C(\overline{\Omega })$ satisfying condition:
	\[p(x) \leq q(x)\leq p^*(x),\ \ \forall \, x \in \Omega.\]
	Moreover, if $\Omega$ is bounded and $q(x)< p^*(x)$ for all $x \in \overline{\Omega}$, the above imbedding is compact.
\end{proposition}

\medskip
\subsection{Variable exponent Orlicz--Sobolev spaces}\label{Orlicz space} ${}$

\smallskip
In this subsection we review the definition and properties of the variable exponent Orlicz--Sobolev space introduced in \cite{Zhang-Radulescu.2018}, which is the solution space for problem ~\eqref{Prob} and the  construction of energy functional.

Throughout this subsection, we assume $(A4)$, $(V1)$ and $(P1)$. Let $\Omega$ be an open domain in $\4{R}^N$. We define the following linear space:
\begin{align*}\label{Sum.Lp,q}
	L^{p(\cdot)}(\Omega)+L^{q(\cdot)}(\Omega):=\{u:u=v+w,v \in L^{p(\cdot)}(\Omega),w\in L^{q(\cdot)}(\Omega)\},
\end{align*}
which is endowed with the norm:
\begin{equation*}\label{def.norm.Lp+Lq}
	|u|_{L^{p(\cdot)}(\Omega)+L^{q(\cdot)}(\Omega)}:=\inf\left\{|v|_{L^{p(\cdot)}(\Omega)}+|w|_{L^{q(\cdot)}(\Omega)}: \ u=v+w,\ v \in L^{p(\cdot)}(\Omega),w\in L^{q(\cdot)}(\Omega)\right\}.
\end{equation*}
%If $\Omega =\4{R}^N$, we use notations $L^{p(\cdot)}(\R^N)+L^{q(\cdot)}(\R^N)$ and $|\cdot|_{L^{p(\cdot)}(\R^N)+L^{q(\cdot)}(\R^N)}$ in the place of $L^{p(\cdot)}(\Omega)+L^{q(\cdot)}(\Omega)$ and $|\cdot|_{L^{p(\cdot)}(\Omega)+L^{q(\cdot)}(\Omega)}$, respectively.

%Throughout this paper, we denote:
%\[\Lambda_{u}:=\{x\in \Omega:|u(x)|>1\},\quad \Lambda_u^c:=\{x \in \Omega: |u(x)| \leq 1\}.\]

%\begin{proposition}[\cite{Ba-Pi-Ro2011,Zhang-Radulescu.2018}]\label{proposition2.6}
%	Assume that $(A4)$ holds and let $u$ be a  function in $L^{p(\cdot)}(\Omega)+L^{q(\cdot)}(\Omega)$. The following statements are true:
	%\begin{enumerate}[(i)]
	%	\item $|\Lambda_u|<+\infty$;
	%	\item  $u \in L^{p(\cdot)}(\Lambda_{u})$;
	%	\item  $u \in L^{q(\cdot)}(\Lambda_u^c) $ ;
	%	\item  The infimum \eqref{def.norm.Lp+Lq} is attained;
	%	\item  If $B \subset \Omega$, then $|u|_{L^{p(\cdot)}(\Omega)+L^{q(\cdot)}(\Omega)} \leq |u|_{L^{p(\cdot)}(B)+L^{q(\cdot)}(B)}+|u|_{L^{p(\cdot)}(\Omega \setminus B)+L^{q(\cdot)}(\Omega \setminus B)} $.
		%	\end{enumerate}
%\end{proposition}

\begin{proposition}[\cite{Zhang-Radulescu.2018}]
	The space $\left(L^{p(\cdot)}(\Omega)+L^{q(\cdot)}(\Omega),|\cdot|_{L^{p(\cdot)}(\Omega)+L^{q(\cdot)}(\Omega)}\right)$ is a reflexive Banach space.
\end{proposition}
We define the variable exponent Orlicz--Sobolev space $\W(V,\Omega)$ as
$$\W(V,\Omega):=\left\{u \in L^{\alpha(\cdot)}(V,\Omega): |\nabla u| \in L^{p(\cdot)}(\Omega)+L^{q(\cdot)}(\Omega) \right\},$$
which is equipped with the norm:
\[\|u\|_{\W(V,\Omega)}:=|u|_{L^{\alpha(\cdot)}(V,\Omega)}+\big||\nabla u|\big|_{L^{p(\cdot)}(\Omega)+L^{q(\cdot)}(\Omega)}.\]
We have the following (see \cite[Proposition 3.11 and Theorem 3.14 (ii)]{Zhang-Radulescu.2018}).
\begin{proposition}\label{W}
	$(\W(V,\Omega),\|\cdot\|_{\W(V,\Omega)})$ is a reflexive Banach space. Furthermore, if $\Omega$ is a bounded Lipschitz domain, then $$\W(V,\Omega) \hookrightarrow\hookrightarrow L^{r(\cdot)}(\Omega)$$ for any $r\in C(\overline{\Omega })$ satisfying $1\leq r(x)< p^*(x)$ for all $x \in \overline{\Omega}$.
\end{proposition}
Hereafter, we simply denote $\left(\W(V,\R^N),\|\cdot\|_{\W(V,\R^N)}\right)$ as $\left(X,\|\cdot\|\right)$. We have the  following imbedding results on $X$ (see \cite[Theorem 3.14]{Zhang-Radulescu.2018}).
\begin{proposition}\label{X-Imbeddings}
	The following conclusions hold:
	\begin{itemize}
		\item[$(i)$] $X \hookrightarrow L^{r(\cdot)}(\R^N)$
		for any $r \in C(\R^N)$ satisfying $\alpha (\cdot) \leq r(\cdot) \leq p^\ast (\cdot) $;
				\item [$(ii)$]  assume additionally that $(V3)$ holds, then $X \hookrightarrow \hookrightarrow L^{r(\cdot)}(\R^N)$
				for any $r \in C(\R^N)$ satisfying $\alpha (\cdot) \leq r(\cdot) \ll p^\ast (\cdot) $.
	\end{itemize}
\end{proposition}
We also have the following compact imbedding, which is crucial to get our existence results; its proof can be found in  \cite[Theorem 3.10]{SRRZ}.
\begin{proposition}\label{IV.compact}
	Let $r \in C_+^{0,1}(\mathbb{R}^N)$  and  $\varrho\in L^{\theta(\cdot)}_+(\R^N)$ with $\theta\in C_+(\R^N)$ satisfying $\alpha(\cdot)\leq  \theta'(\cdot)r(\cdot)\leq p^*(\cdot)$ in $\R^N$. Then, it holds that
	\begin{equation*}
		X \hookrightarrow \hookrightarrow L^{r(\cdot) }(\varrho,\mathbb{R}^N ).
	\end{equation*}
\end{proposition}
The next result is from \cite[Corollary 3.13 (iii)]{Zhang-Radulescu.2018} that shows the density of smooth functions with compact support in $X$.
\begin{proposition}\label{Density}
The space $C_c^\infty(\R^N)$ is dense in $X$.
\end{proposition}

Finally, we present a useful estimate regarding the main operator in \eqref{Prob}, which is proved in Appendix and will be frequently employed in the next sections. Define
\begin{equation}\label{Def.A}
	\mathcal{A}(u):=\int_{\R^N}A(x,\nabla u)\diff x+\int_{\R^N}\frac{V(x)}{\alpha(x)}|u|^{\alpha(x)}\diff x, \ u\in X
\end{equation}
and
\begin{equation}\label{Def.p_alpha}
	p_\alpha(x):=\min\{p(x),\alpha(x)\},\ q_\alpha(x):=\max\{q(x),\alpha(x)\},\ x\in \R^N.
\end{equation}
Then, there exist positive constants $\alpha_1, \alpha_2$ such that
\begin{equation}\label{Est.A}
	\alpha_1\min\left\{\|u\|^{p_\alpha^-},\|u\|^{q_\alpha^+}\right\}\leq \mathcal{A}(u)\leq \alpha_2\max\left\{\|u\|^{p_\alpha^-},\|u\|^{q_\alpha^+}\right\},\ \ \forall u\in X.
\end{equation}

\section{The concentration-compactness principles}\label{CCP}

Let $C_c(\mathbb{R}^N)$ be the
set of all continuous functions $u : \mathbb{R}^N \to\mathbb{R}$ whose support is compact, and let $C_0(\mathbb{R}^N)$ be the completion of $C_c(\mathbb{R}^N)$ relative to the supremum norm $|\cdot|_{L^\infty(\R^N)}.$ Let $\mathcal{M}(\R^N)$ be the space of all signed finite Radon measures on $\mathbb{R}^N$ with the total variation norm. We may identify $\mathcal{M}(\R^N)$ with the dual of $C_0(\mathbb{R}^N)$ via the Riesz representation theorem, that is, for each $\mu\in \left[C_0(\mathbb{R}^N)\right]^\ast$ there is a unique element in $\mathcal{M}(\R^N)$, still denoted by $\mu$, such that
$$\langle \mu,f\rangle=\int_{\mathbb{R}^N}f\diff\mu,\quad \forall f\in C_0(\mathbb{R}^N)$$
(see, e.g., \cite[Section 1.3.3]{Fonseca}). We identify $L^1(\mathbb{R}^N)$ with a subspace of $\mathcal{M}(\R^N)$ through the imbedding
$T:\ L^1(\mathbb{R}^N)\to \left[C_0(\mathbb{R}^N)\right]^\ast$  defined by
$$\langle Tu,f\rangle=\int_{\mathbb{R}^N}uf\diff x, \ \ \forall u\in L^1(\mathbb{R}^N), \ \forall f\in C_0(\mathbb{R}^N).$$

Let $p,q,V,\alpha$ verify $(A4)$, $(V1)$, $(P1)$ and let $b\in L^\infty_+(\Bbb R^{N}),t\in C(\R^N)$ satisfy $\alpha(\cdot)\leq t(\cdot)\leq p^\ast(\cdot)$. Then, it holds $X\hookrightarrow L^{t(\cdot)}(b,\R^N)$ in view of Proposition \ref{X-Imbeddings} (i). Thus,
\begin{equation}\label{Sb}
	S_b:=\inf_{u\in X\setminus\{0\} }\frac{\|u\|}{| u|_{L^{t(\cdot)}(b,\R^N)}}\in (0,\infty).
\end{equation}
We now state the main result of this section, that is, a concentration-compactness principle for the variable exponent Orlicz--Sobolev space $X$.
%======================STATEMENT OF THE CCP=========================%
\begin{theorem}\label{T.CCP}
	Let $(A1)$--$(A6)$, $(V1)$, $(V2)$ and $(P1)$ hold. Let $\{u_n\}_{n\in\mathbb{N}}$ be a bounded sequence in
	$X$  such that
	\begin{eqnarray}
	\label{CCP.w-conv}	u_n &\rightharpoonup& u \quad \text{in}\quad  X,\\
		\label{CCP.mu}	A(x,\nabla u_n)+V(x)|u_n|^{\alpha(x)} &\overset{\ast }{\rightharpoonup }&\mu\quad \text{in}\quad \mathcal{M}(\R^N),\\
	\label{CCP.nu}	b|u_n|^{t(x)}&\overset{\ast }{\rightharpoonup }&\nu\quad \text{in}\quad \mathcal{M}(\R^N).
	\end{eqnarray}
	Then, there exist $\{x_i\}_{i\in \mathcal{I}}\subset \mathcal{C}$ of distinct points and $\{\nu_i\}_{i\in \mathcal{I}}, \{\mu_i\}_{i\in \mathcal{I}}\subset (0,\infty),$ where $\mathcal{I}$ is at most countable and $\mathcal{C}$ is given in \eqref{critical_C}, such that
	\begin{gather}
		\nu=b|u|^{t(x)} + \sum_{i\in \mathcal{I}}\nu_i\delta_{x_i},\label{CCP.form.nu}\\
		\mu \geq A(x,\nabla u) + V(x)|u|^{\alpha(x)}+\sum_{i\in \mathcal{I}} \mu_i \delta_{x_i},\label{CCP.form.mu}\\
		S_b \nu_i^{\frac{1}{t(x_i)}} \leq 2\max\left\{ M_1^{-\frac{1}{p(x_i)}},M_1^{-\frac{1}{q(x_i)}}\right\} \max\left\{\mu_i^{\frac{1}{p(x_i)}},\mu_i^{\frac{1}{q(x_i)}}\right\}, \quad \forall i\in \mathcal{I},\label{CCP.nu_mu}
	\end{gather}
where $S_b$ and $M_1$ are given in \eqref{Sb} and \eqref{Est1}, respectively.
\end{theorem}
Note that the preceding result does not provide any information about a possible loss of mass
at infinity. The next theorem expresses this fact in quantitative terms.
\begin{theorem}\label{T.CCP.infinity}
	Let $(A1)$--$(A6)$, $(V1)$ and $(P1)$ hold. Let $u_n \rightharpoonup u$ in $X$ and set
	$$\mu_\infty:=\lim_{R\to\infty}\underset{n\to\infty}{\lim\sup}\int_{\{|x|>R\}}\left[A(x,\nabla u_n)+V(x)|u_n|^{\alpha(x)}\right]\diff x,$$
	$$\nu_\infty:=\lim_{R\to\infty}\underset{n\to\infty}{\lim\sup}\int_{\{|x|>R\}}b(x)|u_n|^{t(x)}\diff x.$$
	Then 	
	\begin{equation}\label{CCP2.mu}
		\underset{n\to\infty}{\lim\sup}\int_{\mathbb{R}^N}\left[A(x,\nabla u_n)+V(x)|u_n|^{\alpha(x)}\right]\diff x=\mu(\mathbb{R}^N)+\mu_\infty,
	\end{equation}
\begin{equation}\label{CCP2.nu}
	\underset{n\to\infty}{\lim\sup}\int_{\mathbb{R}^N}b(x)|u_n|^{t(x)}\diff x=\nu(\mathbb{R}^N)+\nu_\infty.
\end{equation}
	Assume in addition that:
	\begin{itemize}
		\item[$(\mathcal{E}_\infty)$] There exist positive real numbers $p_\infty$, $q_\infty$, $\alpha_\infty$ and $t_\infty$ such that $$\lim_{|x|\to\infty}p(x)=p_\infty,\ \lim_{|x|\to\infty}q(x)=q_\infty,\ \lim_{|x|\to\infty}\alpha(x)=\alpha_\infty, \text{ and } \lim_{|x|\to\infty}t(x)=t_\infty.$$
	\end{itemize}
	Then
	\begin{equation}\label{CCP2.nu_mu}
		S_b\nu_\infty^{\frac{1}{t_\infty}}\\
		\leq \max\left\{1, M_1^{-\frac{1}{p_\infty}}\right\} \left(\mu_\infty^{\frac{1}{p_\infty}}+\mu_\infty^{\frac{1}{q_\infty}}+ \mu_\infty^{\frac{1}{\alpha_\infty}}\right).
	\end{equation}
	%	where
	%	\begin{equation}\label{CCP2.S_R0}
	%	S_{R_0}:=\underset{\phi\in W_0^{1,p(x)}(B_{R_0}^c)\setminus \{0\}}{\inf}
	%	\frac{\big||\nabla\phi|\big|_{L^{p(x)}(w,B_{R_0}^c)}}{| \phi |_{L^{q(\cdot)}(b,B_{R_0}^c)}}.
	%	\end{equation}
\end{theorem}

%=========================AUXILIARY LEMMAS========================%
Before giving a proof of Theorems~\ref{T.CCP} and \ref{T.CCP.infinity}, we review some auxiliary results for Radon measures.

\begin{lemma}[\cite{HKS}] \label{V.convergence}
	Let $\nu,\{\nu_n\}_{n\in\mathbb{N}}$ be nonnegative and finite Radon measures on $\R^N$ such that $\nu_n\overset{\ast }{\rightharpoonup } \nu$ in $\mathcal{M}(\R^N)$. Then, for any $r\in C_{+}(\R^N)$,
	$$
	|\phi|_{L^{r(\cdot)}_{\nu_n}(\R^N)} \to
	|\phi|_{L^{r(\cdot)}_{\nu}(\R^N)}, \quad \forall \phi\in C_c(\R^N).$$
\end{lemma}

\begin{lemma}[\cite{HK.ANA2021}]\label{V.reserveHolder}
	Let $\mu,\nu$ be two nonnegative and finite Radon measures
	on $\R^N$such that
	$$
	|\phi|_{L_\nu^{r(\cdot)}(\R^N)}\leq C\max\left\{|\phi|_{L_\mu^{p(\cdot)}(\R^N)},|\phi|_{L_\mu^{q(\cdot)}(\R^N)}\right\},\ \ \forall \phi\in C_c^\infty(\R^N)
	$$
for some constant $C>0$ and for some $p,q,r\in C_+(\R^N)$ satisfying $\max\{p(\cdot),q(\cdot)\}\ll r(\cdot)$. Then, there exist an at most countable set $\{x_i\}_{i\in \mathcal{I}}$ of distinct points in $\R^N$ and
	$\{\nu_i\}_{i\in \mathcal{I}}\subset (0,\infty)$, such that
	$$
	\nu=\sum_{i\in \mathcal{I}}\nu_i\delta_{x_i}.
	$$
\end{lemma}
The following result is an extension of the Brezis--Lieb Lemma to weighted variable exponent Lebesgue spaces.
\begin{lemma}[\cite{HKS}]\label{V.brezis-lieb}
	Let $\{f_n\}_{n\in\mathbb{N}}$ be a bounded sequence in $L^{r(\cdot)}(m,\R^N)$ ($r\in C_{+}(\R^N)$, $m\in \mathcal{P}_+(\R^N )$) and $f_n(x)\to f(x)$ a.e. $x\in\R^N$. Then $f\in L^{r(\cdot)}(m,\R^N)$ and
	$$
	\lim_{n\to\infty}\int_{\R^N} \left|m|f_n|^{r(x)}
	-m|f_n-f|^{r(x)}-m|f|^{r(x)}\right|\diff x=0.
	$$
\end{lemma}

%========================PROOOF OF THE CCP=========================%

We are now in a position to prove Theorems~\ref{T.CCP} and \ref{T.CCP.infinity}. In the rest of this section, we denote the ball in $\R^N$ centered at $x_0$ with radius $\epsilon$ by $B_\epsilon(x_0)$ and simply write it as $B_\epsilon$ when $x_0$ is the origin. We also denote $B_\epsilon^c=\R^N\setminus B_\epsilon$

\begin{proof}[\textbf{Proof of Theorem~\ref{T.CCP}}]
	
	Let $v_n=u_n-u$. Then, up to a subsequence, we have
	\begin{eqnarray}\label{P_CCP.conv.of.v_n}
		\begin{cases}
			v_n(x) &\to  \quad 0 \ \  \text{a.e.}\ \  x\in\R^N,\\
			v_n &\rightharpoonup \quad 0 \ \  \text{in}\ \  X.
		\end{cases}
	\end{eqnarray}
	So, by Lemma~\ref{V.brezis-lieb}, we deduce that
	$$
	\lim_{n\to\infty}\int_{\R^N}\left|b|u_n|^{t(x)}
	-b|v_n|^{t(x)}-b|u|^{t(x)}\right|\diff x=0.$$
	From this and \eqref{P_CCP.conv.of.v_n}, we easily obtain
	$$
	\lim_{n\to\infty}\Big(\int_{\R^N} \phi b|u_n|^{t(x)}\diff x
	-\int_{\R^N}\phi b|v_n|^{t(x)} \diff x\Big)
	=\int_{\R^N} \phi b|u|^{t(x)} \diff x, \quad \forall  \phi\in C_0({\R^N}),$$
	i.e.,
	\begin{equation}\label{P_CCP.w*-vn}
		b|v_n|^{t(x)}\overset{\ast }{\rightharpoonup }\bar{\nu}=\nu-b|u|^{t(x)}\quad \text{in} \ \ \mathcal{M}(\R^N).
	\end{equation}
	By \eqref{Est.A}, it is clear that $\left\{A(x,\nabla v_n)+V(x)|v_n|^{\alpha(x)} \right\}_{n\in\mathbb{N}}$ is bounded in $L^1(\R^N)$. So up to a subsequence, we have
	\begin{equation}\label{P_CCP.w*-bar.mu}
		A(x,\nabla v_n)+V(x)|v_n|^{\alpha(x)} \overset{\ast }{\rightharpoonup }\quad \bar{\mu}\quad \text{in}\ \ \mathcal{M}(\R^N)
	\end{equation}
	for some finite nonnegative Radon measure $\bar{\mu}$ on $\R^N$. By the definition of $X$, it is easy to see that $\phi v\in X$ for any $\phi\in C_c^\infty(\R^N)$ with $\operatorname{supp}(\phi)\subset B_R$ and for any $v\in X$. So, utilizing \eqref{Sb}, for any $\phi\in C_c^\infty(\R^N),$ we get
	\begin{align}\label{P_CCP.*}
		\notag&S_b|\phi v_n|_{L^{t(\cdot)}(b,\R^N)}\\
		\notag &\leq \big||\nabla(\phi v_n)|\big|_{L^{p(\cdot)}(\R^N)+L^{q(\cdot)}(\R^N)}+|\phi v_n|_{L^{\alpha(\cdot)}(V,\R^N)}\notag\\
		&\leq \big||\phi\nabla v_n|\big|_{L^{p(\cdot)}(\R^N)+L^{q(\cdot)}(\R^N)}+\big||v_n\nabla\phi|\big|_{L^{p(\cdot)}(\R^N)+L^{q(\cdot)}(\R^N)}+|\phi v_n|_{L^{\alpha(\cdot)}(V,\R^N)}.
	\end{align}
We will prove that
\begin{equation}\label{P_CCP.x}
\limsup_{n\to\infty}\big||v_n\nabla\phi|\big|_{L^{p(\cdot)}(\R^N)+L^{q(\cdot)}(\R^N)}=0,
\end{equation}
\begin{equation}\label{P_CCP.y}
	\limsup_{n\to\infty}|\phi v_n|_{L^{\alpha(\cdot)}(V,\R^N)}=0,
\end{equation}
and
\begin{equation}\label{P_CCP.a}
	\limsup_{n\to\infty}\big||\phi\nabla v_n|\big|_{L^{p(\cdot)}(\R^N)+L^{q(\cdot)}(\R^N)}
	\leq 2\max\left\{1,M_1^{-\frac{1}{p^-}}\right\} \max\left\{|\phi|_{L^{p(\cdot)}_{\bar{\mu}}(\R^N)},|\phi|_{L^{q(\cdot)}_{\bar{\mu}}(\R^N)}\right\},
\end{equation}
where $M_1$ is given in \eqref{Est1}.

Assuming \eqref{P_CCP.x}-\eqref{P_CCP.a} for the moment, by taking the limit superior as $n\to \infty$ in \eqref{P_CCP.*} and  invoking Lemma~\ref{V.convergence} we obtain
\begin{equation*}
	S_b|\phi|_{L_{\bar{\nu}}^{t(\cdot)}(\R^N)}
	\leq 2\max\left\{1,M_1^{-\frac{1}{p^-}}\right\} \max\left\{|\phi|_{L^{p(\cdot)}_{\bar{\mu}}(\R^N)},|\phi|_{L^{q(\cdot)}_{\bar{\mu}}(\R^N)}\right\};
\end{equation*}
hence, \eqref{CCP.form.nu} follows in view of Lemma~\ref{V.reserveHolder} and \eqref{P_CCP.w*-vn}.

To see \eqref{P_CCP.x}, we first notice that
\begin{align*}
v_n\nabla\phi=v_n\nabla\phi\chi_{\{|v_n\nabla\phi|\leq 1\}}+v_n\nabla\phi\chi_{\{|v_n\nabla\phi|\geq 1\}}.
\end{align*}
From this and the facts that $v_n\nabla\phi\chi_{\{|v_n\nabla\phi|\leq 1\}}\in \left(L^{p(\cdot)}(\R^N)\right)^N$ and $v_n\nabla\phi\chi_{\{|v_n\nabla\phi|\geq 1\}}\in \left(L^{q(\cdot)}(\R^N)\right)^N$ we deduce that
\begin{align*}
	\big||v_n\nabla\phi|\big|_{L^{p(\cdot)}(\R^N)+L^{q(\cdot)}(\R^N)}&\leq \big||v_n\nabla\phi\chi_{\{|v_n\nabla\phi|\leq 1\}}|\big|_{L^{p(\cdot)}(\R^N)}+\big||v_n\nabla\phi\chi_{\{|v_n\nabla\phi|\geq 1\}}|\big|_{L^{q(\cdot)}(\R^N)}\\
	&\leq \big||v_n\nabla\phi|\big|_{L^{p(\cdot)}(B_R)}+\big||v_n\nabla\phi|\big|_{L^{q(\cdot)}(B_R)}\\
	&\leq \big||\nabla\phi|\big|_{L^{\infty}(\R^N)}\left(|v_n|_{L^{p(\cdot)}(B_R)}+|v_n|_{L^{q(\cdot)}(B_R)}\right).
\end{align*}
Then we obtain \eqref{P_CCP.x} by invoking Proposition~\ref{W} and \eqref{P_CCP.conv.of.v_n}.

To see \eqref{P_CCP.y}, we note that
\begin{equation*}
	|\phi v_n|_{L^{\alpha(\cdot)}(V,\R^N)}=|\phi v_n|_{L^{\alpha(\cdot)}(V,B_R)}\leq |\phi|_{L^\infty(\R^N)}\, |v_n|_{L^{\alpha(\cdot)}(B_R)}.
\end{equation*}
By this, we obtain \eqref{P_CCP.y} in view of Proposition~\ref{W}.

Finally, we prove \eqref{P_CCP.a}. To this end, we first note that by \eqref{Est1},
\begin{equation*}
	\int_{\R^N}|\phi\nabla v_n\chi_{\{|\nabla v_n|\geq 1\}}|^{p(x)}\diff x\leq C_q\int_{\R^N}|\nabla v_n|^{p(x)}\chi_{\{|\nabla v_n|\geq 1\}}\diff x\leq C_qM_1^{-1}\int_{\R^N}A(x,\nabla v_n)\diff x<\infty
\end{equation*}
and
\begin{equation*}
	\int_{\R^N}|\phi\nabla v_n\chi_{\{|\nabla v_n|\leq 1\}}|^{q(x)}\diff x\leq C_q\int_{\R^N}|\nabla v_n|^{q(x)}\chi_{\{|\nabla v_n|\leq 1\}}\diff x\leq C_qM_1^{-1}\int_{\R^N}A(x,\nabla v_n)\diff x<\infty,
\end{equation*}
where $C_q:=1+|\phi|_{L^\infty(\R^N)}^{q^+}$. Hence, $\phi\nabla v_n\chi_{\{|\nabla v_n|\leq 1\}}\in \left(L^{q(\cdot)}(\R^N)\right)^N$ and $\phi\nabla v_n\chi_{\{|\nabla v_n|\geq 1\}}\in \left(L^{p(\cdot)}(\R^N)\right)^N$. From this and the fact that
\begin{equation*}
	\phi\nabla v_n=\phi\nabla v_n\chi_{\{|\nabla v_n|\leq 1\}}+\phi\nabla v_n\chi_{\{|\nabla v_n|\geq 1\}},
\end{equation*}
we get that
\begin{equation}\label{P_CCP.lambda_n-E}
	\lambda_n:=\big||\phi\nabla v_n|\big|_{L^{p(\cdot)}(\R^N)+L^{q(\cdot)}(\R^N)}\leq \lambda_n^{1}+\lambda_n^{2},
\end{equation}
where \begin{equation}\label{P_CCP.lambda_n-E*}
	\lambda_n^1:= \big||\phi\nabla v_n\chi_{\{|\nabla v_n|\geq 1\}}|\big|_{L^{p(\cdot)}(B_R)}\ \ \text{and}\ \ \lambda_n^2:=\big||\phi\nabla v_n\chi_{\{|\nabla v_n|\leq 1\}}|\big|_{L^{q(\cdot)}(B_R)}.
\end{equation}
Up to a subsequence, we may assume that
\begin{equation}\label{P_CCP.lambda_n}
	\lim_{n\to\infty}\lambda_n= \limsup_{n\to\infty}\lambda_n=:\lambda_*
	,\ \lim_{n\to\infty}\lambda_n^i= \limsup_{n\to\infty}\lambda_n^i=:\lambda_*^i\ (i=1,2).
\end{equation}
Since  $\{\lambda_n\}_{n\in\mathbb{N}}$, $\{\lambda_n^1\}_{n\in\mathbb{N}}$ and $\{\lambda_n^2\}_{n\in\mathbb{N}}$ are bounded sequences, we have $\lambda_*, \lambda_*^1, \lambda_*^2\in [0,\infty)$. Obviously, \eqref{P_CCP.a} holds for the case $\lambda_*=0$. Let us consider the case $\lambda_*>0$. From \eqref{P_CCP.lambda_n-E} we have $\lambda_*\leq \lambda_*^1+\lambda_*^2$. Hence, we obtain
\begin{equation}\label{P_CCP.a*}
	\lambda_*\leq 2\max\left\{\lambda_*^1,\lambda_*^2\right\}.
\end{equation}
Thus, \eqref{P_CCP.a} immediately follows if we can show that
\begin{equation}\label{P_CCP.b}
	\lambda_*^1\leq \max\left\{1,M_1^{-\frac{1}{p^-}}\right\}|\phi|_{L^{p(\cdot)}_{\bar{\mu}}(\R^N)}
\end{equation}
and
\begin{equation}\label{P_CCP.c}
\lambda_*^2\leq \max\left\{1,M_1^{-\frac{1}{q^-}}\right\}|\phi|_{L^{q(\cdot)}_{\bar{\mu}}(\R^N)}.
\end{equation}
We only prove \eqref{P_CCP.b} since \eqref{P_CCP.c} can be proved similarly. Clearly, \eqref{P_CCP.b} holds for the case $\lambda_*^1=0$. For the case $\lambda_*^1>0$, using \eqref{Est1} we have that for $n$ large,
	\begin{align*}
		1&=\int_{\R^N} \left|\frac{\phi\nabla v_n\chi_{\{|\nabla v_n|\geq 1\}}}{\lambda_n^1}\right|^{p(x)}\diff x\\
&\leq \int_{\R^N }\left|\frac{\phi}{\lambda_n^1}\right|^{p(x)}M_1^{-1} A(x,\nabla v_n)\diff x\\
		&\leq \int_{\R^N }\left|\frac{\phi}{\min\left\{1,M_1^{\frac{1}{p^-}}\right\}\lambda_n^1}\right|^{p(x)} A(x,\nabla v_n)\diff x\\
&\leq \int_{\R^N }\left|\frac{\phi}{\min\left\{1,M_1^{\frac{1}{p^-}}\right\}\lambda_n^1}\right|^{p(x)} \round{A(x,\nabla v_n) + V(x)|v_n|^{\alpha(x)}}\diff x.
	\end{align*}
Taking the limit as $n\to\infty$ in the last equality and using \eqref{P_CCP.w*-bar.mu} we obtain
	\begin{equation*}
		1\leq \int_{\R^N} \left|\frac{\phi}{\min\left\{1,M_1^{\frac{1}{p^-}}\right\}\lambda_*^1}\right|^{p(x)}\diff \bar{\mu}.
	\end{equation*}
	Equivalently,
	\begin{equation*}
		\lambda_*^1\leq \max\left\{1,M_1^{-\frac{1}{p^-}}\right\}|\phi|_{L_{\bar{\mu}}^{p(\cdot)}(\R^N)}.
	\end{equation*}
Thus, we have proved \eqref{P_CCP.b} (and similarly, \eqref{P_CCP.c}); hence, \eqref{P_CCP.a} follows due to \eqref{P_CCP.a*}.

We claim that $\{x_i\}_{i\in \mathcal{I}}\subset \mathcal{C}$. Assume by contradiction that there is some $x_i\in \R^N\setminus\mathcal{C}$. Let $\delta>0$ be such that $\overline{B_{2\delta}(x_i)}\subset \mathbb{R}^N\setminus\mathcal{C}$. Set $B=B_\delta(x_i)\cap \R^N$ then $\overline{B}\subset \R^N\setminus\mathcal{C}$ and hence, $t(x)<p^\ast(x)$ for all $x\in\overline{B}$. Thus, $u_n\to u$ in $L^{t(\cdot)}(B)$ in view of Proposition~\ref{W}. Equivalently, $\int_{B}|u_n-u|^{t(x)}\diff x\to 0$ due to Proposition~\ref{norm-modular}; hence, $\int_{B}b|u_n-u|^{t(x)}\diff x\to 0$ since
	$$\int_{B}b|u_n-u|^{t(x)}\diff x\leq |b|_{L^\infty(\R^N)}\int_{B}|u_n-u|^{t(x)}\diff x.$$
	Thus, invoking Lemma~\ref{V.brezis-lieb} again, we infer
	$$\int_{B}b|u_n|^{t(x)}\diff x\to \int_{B}b|u|^{t(x)}\diff x.$$
	From this and the fact that $\nu (B)\leq \liminf_{n\to \infty}\int_{B}b|u_n|^{t(x)}\diff x$ (see \cite[Proposition 1.203]{Fonseca}), we obtain $\nu (B)\leq \int_{B}b|u|^{t(x)}\diff x.$
	Meanwhile, from \eqref{CCP.form.nu}, we have
	$$\nu (B)\geq \int_{B}b|u|^{t(x)}\diff x+\nu_i>\int_{B}b|u|^{t(x)}\diff x,$$
	a contradiction. So $\{x_i\}_{i\in \mathcal{I}}\subset \mathcal{C}$.

	Next, to obtain \eqref{CCP.nu_mu}, let $\eta$ be in $C_c^\infty(\mathbb{R}^N)$ such that $0\leq \eta\leq 1,$ $\eta\equiv 1$ on $B_{\frac{1}{2}}(0)$ and $\eta\equiv 0$ outside $B_1(0)$. Fix $i\in \mathcal{I}$ and take $\epsilon>0$.  Set $\phi_{i,\epsilon}(x):=
	\eta(\frac{x-x_i}{\epsilon})$ and
	\begin{gather*}
		p^+_{i,\epsilon}:=\sup_{x\in B_\epsilon(x_i)}p(x),\quad
		p^-_{i,\epsilon}:=\inf_{x\in B_\epsilon(x_i)}p(x),\\
		t^+_{i,\epsilon}:=\sup_{x\in B_\epsilon(x_i)}t(x),\quad
		t^-_{i,\epsilon}:=\inf_{x\in B_\epsilon(x_i)}t(x).
	\end{gather*}
		Thus by \eqref{Sb}, we have
\begin{align*}
	\notag S_b&|\phi_{i,\epsilon} u_n|_{L^{t(\cdot)}(b,\R^N)}\\
	\notag \leq& \big||\nabla(\phi_{i,\epsilon} u_n)|\big|_{L^{p(\cdot)}(\R^N)+L^{q(\cdot)}(\R^N)}+|\phi_{i,\epsilon} u_n|_{L^{\alpha(\cdot)}(V,\R^N)}\\
	\leq &\big||\phi_{i,\epsilon}\nabla u_n|\big|_{L^{p(\cdot)}(\R^N)+L^{q(\cdot)}(\R^N)}+|\phi_{i,\epsilon} u_n|_{L^{\alpha(\cdot)}(V,\R^N)}+\big||u_n\nabla\phi_{i,\epsilon}|\big|_{L^{p(\cdot)}(\R^N)+L^{q(\cdot)}(\R^N)}.
\end{align*}
	Taking the limit superior as $n\to \infty$ and  invoking Lemma~\ref{V.convergence}, then taking the limit superior as $\epsilon\to 0^+$ in the last inequality  we obtain
	\begin{multline}\label{P_CCP.est1}
		S_b \limsup_{\epsilon\to 0^+}|\phi_{i,\epsilon}|_{L^{t(\cdot)}_\nu(\R^N)}\leq \limsup_{\epsilon\to 0^+}\limsup_{n\to\infty} \big||\phi_{i,\epsilon}\nabla u_n|\big|_{L^{p(\cdot)}(\R^N)+L^{q(\cdot)}(\R^N)}\\+\limsup_{\epsilon\to 0^+}\limsup_{n\to\infty}|\phi_{i,\epsilon} u_n|_{L^{\alpha(\cdot)}(V,\R^N)}
		+\limsup_{\epsilon\to 0^+}\limsup_{n\to\infty}\big||u_n\nabla\phi_{i,\epsilon}|\big|_{L^{p(\cdot)}(\R^N)+L^{q(\cdot)}(\R^N)}.
	\end{multline}
	From Proposition~\ref{norm-modular}, we have
	$$|\phi_{i,\epsilon}|_{L^{t(\cdot)}_\nu(\R^N)}\geq \min\left\{\left(\int_{B_\epsilon(x_i)}|\phi_{i,\epsilon}|^{t(x)}\diff\nu\right)^{\frac{1}{t^+_{i,\epsilon}}} ,\left(\int_{B_\epsilon(x_i)}|\phi_{i,\epsilon}|^{t(x)}\diff\nu\right)^{\frac{1}{t^-_{i,\epsilon}}}\right\}.$$
	It is clear that
	$$\int_{B_\epsilon(x_i)}|\phi_{i,\epsilon}|^{t(x)}\diff\nu\geq
	\int_{B_{\frac{\epsilon}{2}}(x_i)}|
	\phi_{i,\epsilon}|^{t(x)}\diff\nu\geq \nu_i.$$
	Thus, we obtain
	\begin{equation*}
		|\phi_{i,\epsilon}|_{L^{t(\cdot)}_\nu(\R^N)}\geq \min\left\{\nu_i^{\frac{1}{t^+_{i,\epsilon}}} ,\nu_i^{\frac{1}{t^-_{i,\epsilon}}}\right\}.
	\end{equation*}
Hence,
\begin{equation}\label{P_CCP.est2}
	\limsup_{\epsilon\to 0^+}|\phi_{i,\epsilon}|_{L^{t(\cdot)}_\nu(\R^N)}\geq \nu_i^{\frac{1}{t(x_i)}}.
\end{equation}
	On the other hand, arguing as that obtained \eqref{P_CCP.lambda_n-E} we get that
	\begin{multline}\label{P_CCP.Grad1}
		\big||\phi_{i,\epsilon}\nabla u_{n}|\big|_{L^{p(\cdot)}(\R^N)+L^{q(\cdot)}(\R^N)}\\
		\leq \big||\phi_{i,\epsilon}\nabla u_{n}\chi_{\{|\nabla u_{n}|\geq 1\}}|\big|_{L^{p(\cdot)}(B_\epsilon(x_i))}+\big||\phi_{i,\epsilon}\nabla u_{n}\chi_{\{|\nabla u_{n}|\leq 1\}}|\big|_{L^{q(\cdot)}(B_\epsilon(x_i))}.
	\end{multline}
	Invoking Proposition~\ref{norm-modular}, we have
	\begin{align}\label{P_CCP.Grad2}
		\notag&\big||\phi_{i,\epsilon}\nabla u_{n}\chi_{\{|\nabla u_{n}|\geq 1\}}|\big|_{L^{p(\cdot)}(B_\epsilon(x_i))}\\
		\notag&\leq\max\left\{\left(\int_{B_\epsilon(x_i)}|\nabla u_{n}|^{p(x)}\chi_{\{|\nabla u_{n}|\geq 1\}}\phi_{i,\epsilon}^{p(x)}\diff x\right)^{\frac{1}{p^+_{i,\epsilon}}},\left(\int_{B_\epsilon(x_i)}|\nabla u_{n}|^{p(x)}\chi_{\{|\nabla u_{n}|\geq 1\}}\phi_{i,\epsilon}^{p(x)}\diff x\right)^{\frac{1}{p^-_{i,\epsilon}}}\right\} \\
		\notag	&\leq\max\left\{\left(\int_{B_\epsilon(x_i)}M_1^{-1}A(x,\nabla u_{n})\phi_{i,\epsilon}\diff x\right)^{\frac{1}{p^+_{i,\epsilon}}},\left(\int_{B_\epsilon(x_i)}M_1^{-1}A(x,\nabla u_{n})\phi_{i,\epsilon}\diff x\right)^{\frac{1}{p^-_{i,\epsilon}}}\right\} \\
		\notag	\notag&\leq\max\left\{M_1^{-\frac{1}{p^+_{i,\epsilon}}},M_1^{-\frac{1}{p^-_{i,\epsilon}}}\right\}\max\biggl\{\left(\int_{B_\epsilon(x_i)}\left[A(x,\nabla u_{n})+V(x)|u_{n}|^{\alpha(x)}\right]\phi_{i,\epsilon}\diff x\right)^{\frac{1}{p^+_{i,\epsilon}}},\\
		&{}\hspace*{6cm}\left(\int_{B_\epsilon(x_i)}\left[A(x,\nabla u_{n})+V(x)|u_{n}|^{\alpha(x)}\right]\phi_{i,\epsilon}\diff x\right)^{\frac{1}{p^-_{i,\epsilon}}}\biggr\}.
	\end{align}
	Similarly we have
	\begin{align}\label{P_CCP.Grad3}
		\notag&\big||\phi_{i,\epsilon}\nabla u_{n}\chi_{\{|\nabla u_{n}|\leq 1\}}|\big|_{L^{q(\cdot)}(B_\epsilon(x_i))}\\
		\notag	&\leq\max\left\{\left(\int_{B_\epsilon(x_i)}|\nabla u_{n}|^{q(x)}\phi_{i,\epsilon}^{q(x)}\chi_{\{|\nabla u_{n}|\leq 1\}}\diff x\right)^{\frac{1}{q^+_{i,\epsilon}}},\left(\int_{B_\epsilon(x_i)}|\nabla u_{n}|^{q(x)}\phi_{i,\epsilon}^{q(x)}\chi_{\{|\nabla u_{n}|\leq 1\}}\diff x\right)^{\frac{1}{q^-_{i,\epsilon}}}\right\} \\
		\notag	&\leq\max\left\{M_1^{-\frac{1}{q^+_{i,\epsilon}}},M_1^{-\frac{1}{q^-_{i,\epsilon}}}\right\}\max\biggl\{\left(\int_{B_\epsilon(x_i)}\left[A(x,\nabla u_{n})+V(x)|u_{n}|^{\alpha(x)}\right]\phi_{i,\epsilon}\diff x\right)^{\frac{1}{q^+_{i,\epsilon}}},\\
		&{}\hspace*{6cm}\left(\int_{B_\epsilon(x_i)}\left[A(x,\nabla u_{n})+V(x)|u_{n}|^{\alpha(x)}\right]\phi_{i,\epsilon}\diff x\right)^{\frac{1}{q^-_{i,\epsilon}}}\biggr\}.
	\end{align}
		From \eqref{P_CCP.Grad1}--\eqref{P_CCP.Grad3} we obtain
	\begin{multline*}
		\limsup_{n\to\infty}\big||\phi_{i,\epsilon}\nabla u_{n}|\big|_{L^{p(\cdot)}(\R^N)+L^{q(\cdot)}(\R^N)}\\
		\leq2\max\left\{ M_1^{-\frac{1}{p^-_{i,\epsilon}}},M_1^{-\frac{1}{q^+_{i,\epsilon}}}\right\} \max\left\{\mu(B_\epsilon(x_i))^{\frac{1}{p^-_{i,\epsilon}}},\mu(B_\epsilon(x_i))^{\frac{1}{q^+_{i,\epsilon}}}\right\}.
	\end{multline*}
Thus, one has
\begin{multline}\label{P_CCP.est3}
\limsup_{\epsilon\to 0^+}	\limsup_{n\to\infty}\big||\phi_{i,\epsilon}\nabla u_{n}|\big|_{L^{p(\cdot)}(\R^N)+L^{q(\cdot)}(\R^N)}\\
	\leq 2\max\left\{ M_1^{-\frac{1}{p(x_i)}},M_1^{-\frac{1}{q(x_i)}}\right\} \max\left\{\mu_i^{\frac{1}{p(x_i)}},\mu_i^{\frac{1}{q(x_i)}}\right\},
\end{multline}
where $\mu_i:=\lim_{\epsilon\to 0^+}\mu(B_\epsilon(x_i)$. We have
\begin{align*}
	\notag &\big|\phi_{i,\epsilon} u_{n}\big|_{L^{\alpha(\cdot)}(V,\R^N)}\leq\max\left\{|V|_{L^\infty(B_1(x_i))}^{-\frac{1}{\alpha^+}}\,|V|_{L^\infty(B_1(x_i))}^{-\frac{1}{\alpha^-}}\right\}|u_{n}|_{L^{\alpha(\cdot)}(B_\epsilon(x_i))}
\end{align*}
for all $n\in\N$ and all $\epsilon\in (0,1)$. Thus, by invoking Proposition~\ref{W} we obtain
\begin{equation}\label{P_CCP.est4}
	\limsup_{\epsilon\to 0^+}	\limsup_{n\to\infty} \big|\phi_{i,\epsilon} u_{n}\big|_{L^{p(\cdot)}(V,\R^N)}=0.
\end{equation}
Now we analyze the last term in \eqref{P_CCP.est1} to get \eqref{CCP.nu_mu}. First we have
	\begin{align*}
		\notag\big||u_n\nabla\phi_{i,\epsilon}|&\big|_{L^{p(\cdot)}(\R^N)+L^{q(\cdot)}(\R^N)}\\
		&\leq \big||u_n\nabla\phi_{i,\epsilon}\chi_{\{|u_n\nabla\phi_{i,\epsilon}|\geq 1\}}|\big|_{L^{p(\cdot)}(\R^N)}+\big||u_n\nabla\phi_{i,\epsilon}\chi_{\{|u_n\nabla\phi_{i,\epsilon}|\leq 1\}}|\big|_{L^{q(\cdot)}(\R^N)}.
	\end{align*}
	That is,
	\begin{multline}\label{P_CCP.est5*}
		\big||u_n\nabla\phi_{i,\epsilon}|\big|_{L^{p(\cdot)}(\R^N)+L^{q(\cdot)}(\R^N)}\\
		\leq \big||u_n\nabla\phi_{i,\epsilon}\chi_{\{|u_n\nabla\phi_{i,\epsilon}|\geq 1\}}|\big|_{L^{p(\cdot)}(B_\epsilon(x_i))}+\big||u_n\nabla\phi_{i,\epsilon}\chi_{\{|u_n\nabla\phi_{i,\epsilon}|\leq 1\}}|\big|_{L^{q(\cdot)}(B_\epsilon(x_i))}.
	\end{multline}% Since $W_V^{1,p(x)}(B_{2R}\setminus \overline{B_R})=W^{1,p(x)}(B_{2R}\setminus \overline{B_R})\hookrightarrow\hookrightarrow L^{p(x)}(B_{2R}\setminus \overline{B_R}),$ we deduce
	By Proposition~\ref{W} again, we have that $X\hookrightarrow\hookrightarrow L^{p(\cdot)}(B_\epsilon(x_i))$ and $X\hookrightarrow\hookrightarrow L^{q(\cdot)}(B_\epsilon(x_i)).$ Hence,
	\begin{equation}\label{P_CCP.est5'}
		\underset{{n}\to\infty}{\lim\sup}\big||u_n\nabla\phi_{i,\epsilon}\chi_{\{|u_n\nabla\phi_{i,\epsilon}|\geq 1\}}|\big|_{L^{p(\cdot)}(B_\epsilon(x_i))}=\big||u\nabla\phi_{i,\epsilon}\chi_{\{|u\nabla\phi_{i,\epsilon}|\geq 1\}}|\big|_{L^{p(\cdot)}(B_\epsilon(x_i))}
	\end{equation}
	and
	\begin{equation}\label{P_CCP.est5''}
		\underset{{n}\to\infty}{\lim\sup}\big||u_n\nabla\phi_{i,\epsilon}\chi_{\{|u_n\nabla\phi_{i,\epsilon}|\leq 1\}}|\big|_{L^{q(\cdot)}(B_\epsilon(x_i))}=\big||u\nabla\phi_{i,\epsilon}\chi_{\{|u\nabla\phi_{i,\epsilon}|\leq 1\}}|\big|_{L^{q(\cdot)}(B_\epsilon(x_i))}.
	\end{equation}
Also we note that
	\begin{multline}\label{P_CCP.est5'''}
	\max\left\{\int_{B_\epsilon(x_i)}\big|u\nabla\phi_{i,\epsilon}\chi_{\{|u\nabla\phi_{i,\epsilon}|\geq 1\}}\big|^{p(x)}\diff x, 	\int_{B_\epsilon(x_i)}\big|u\nabla\phi_{i,\epsilon}\chi_{\{|u\nabla\phi_{i,\epsilon}|\leq 1\}}\big|^{q(x)}\diff x\right\}\\
		\leq	\int_{B_\epsilon(x_i)}\big|u\nabla\phi_{i,\epsilon}\big|^{p(x)}\diff x.
	\end{multline}
	Applying Proposition~\ref{Holder}, using the fact $u\in L^{p^\ast(\cdot)}(\mathbb{R}^N)$ (see Proposition~\ref{X-Imbeddings}) we obtain
	\begin{align*}
		\notag\int_{B_\epsilon(x_i)}\big|u\nabla\phi_{i,\epsilon}\big|^{p(x)}\diff x\leq 2\big||u|^{p(\cdot)}\big|_{L^{\frac{p^\ast(\cdot)}{p(\cdot)}}(B_\epsilon(x_i))}\big||\nabla\phi_{i,\epsilon}|^{p(\cdot)}\big|_{L^{\frac{N}{p(\cdot)}}(B_\epsilon(x_i))}\leq 2c \big||u|^{p(\cdot)}\big|_{L^{\frac{p^\ast(\cdot)}{p(\cdot)}}(B_\epsilon(x_i))},
	\end{align*}
	where $c$ is a positive constant. Here in view of Proposition \ref{norm-modular} we have used the fact that
	$$\big||\nabla\phi_{i,\epsilon}|^{p(\cdot)}\big|_{L^{\frac{N}{p(\cdot)}}(B_\epsilon(x_i))}\leq \left(1+\int_{B_\epsilon(x_i)}|\nabla\phi_{i,\epsilon}|^N\diff x \right)^{\frac{p^+}{N}}\leq \text{constant},\ \forall \epsilon>0.$$
	Hence, the estimate \eqref{P_CCP.est5'''} infers
	$$\lim_{\epsilon\to 0^+}\int_{B_\epsilon(x_i)}\big|u\nabla\phi_{i,\epsilon}\chi_{\{|u\nabla\phi_{i,\epsilon}|\geq 1\}}\big|^{p(x)}\diff x=0$$
	and
	$$\lim_{\epsilon\to 0^+}\int_{B_\epsilon(x_i)}\big|u\nabla\phi_{i,\epsilon}\chi_{\{|u\nabla\phi_{i,\epsilon}|\leq 1\}}\big|^{q(x)}\diff x=0.$$
	Hence, one has
	\begin{equation*}
		\lim_{\epsilon\to 0^+}\big||u\nabla\phi_{i,\epsilon}\chi_{\{|u\nabla\phi_{i,\epsilon}|\geq 1\}}|\big|_{L^{p(\cdot)}(B_\epsilon(x_i))}=	\lim_{\epsilon\to 0^+}\big||u\nabla\phi_{i,\epsilon}\chi_{\{|u\nabla\phi_{i,\epsilon}|\leq 1\}}|\big|_{L^{q(\cdot)}(B_\epsilon(x_i))}=0.
	\end{equation*}
	From this and \eqref{P_CCP.est5'}--\eqref{P_CCP.est5''} we deduce from \eqref{P_CCP.est5*} that
	\begin{equation}\label{P_CCP.est5}
		\lim_{\epsilon\to 0^+}\underset{n\to\infty}{\lim\sup}\big||u_n\nabla\phi_{i,\epsilon}|\big|_{L^{p(\cdot)}(\R^N)+L^{q(\cdot)}(\R^N)}=0.
	\end{equation}
	Utilizing \eqref{P_CCP.est1}, \eqref{P_CCP.est2}, \eqref{P_CCP.est3}, \eqref{P_CCP.est4} and \eqref{P_CCP.est5} we obtain \eqref{CCP.nu_mu}.

	Finally, to show \eqref{CCP.form.mu}, note that for any $\phi\in C(\R^N)$  with $\phi\geq 0$, the functional $u\mapsto \Phi(u):=\int_{\R^N}\phi(x)\big(A(x,\nabla u) + V(x)|u|^{\alpha(x)}\big)\diff x$ is G\^ateaux differentiable on $X$ due to $(A1)$ with
	$$\langle \Phi'(u),v\rangle=\int_{\R^N}\phi(x)a(x,\nabla u)\cdot \nabla v\diff x+\int_{\R^N}\phi(x)\alpha(x)V(x)|u|^{\alpha(x)-2}uv\diff x,\quad \forall u,v\in X.$$
	Here and in the sequel, by $\langle \cdot,\cdot\rangle$ we denote the duality pairing between $X$ and its dual $X^*$. By $(A6)$, $\Phi': X\to X^*$ is strictly monotone and hence, $\Phi$ is weakly lower semicontinuous and therefore,
	\begin{align*}
\int_{\R^N}\phi(x)\big(A(x,\nabla u) + V(x)|u|^{\alpha(x)}\big)\diff x\leq \liminf_{n\to  \infty } \int_{\R^N }\phi(x)\big(A(x,\nabla u_n) + V(x)|u_{n}|^{\alpha(x)}\big)\diff x=\int_{\R^N}\phi \diff\mu.
\end{align*}
	Thus, $\mu\geq A(x,\nabla u) + V(x)|u|^{\alpha(x)}$. By extracting $\mu$ to its atoms, we deduce \eqref{CCP.form.mu} and the proof is complete.
\end{proof}

%=====================PROOF OF THE CCP AT INFINITY==============================%
\begin{proof}[\textbf{Proof of Theorem~\ref{T.CCP.infinity}}]
	Let $\phi$ be in $C^\infty(\mathbb{R})$ such that  $\phi(t)\equiv 0$ on $|t|\leq 1,$ $\phi(t)\equiv 1$ on $|t|\geq 2$ and $0\leq \phi\leq 1,$ $|\phi'|\leq 2$ in $\mathbb{R}.$ For each $R>0$, set $\phi_R(x):=\phi(|x|/R)$ for all $x\in\mathbb{R}^N.$ We then decompose
	\begin{align}\label{P_CCP2.dec}
		\int_{ \mathbb{R}^N}\big[A(x,\nabla u_n)&+V(x)|u_n|^{\alpha(x)}\big]\diff x=\int_{ \mathbb{R}^N}\left[A(x,\nabla u_n)+V(x)|u_n|^{\alpha(x)}\right]\phi_R\diff x\notag\\
		&+\int_{ \mathbb{R}^N}\left[A(x,\nabla u_n)+V(x)|u_n|^{\alpha(x)}\right](1-\phi_R)\diff x.
	\end{align}
	For the first term of the right-hand side of \eqref{P_CCP2.dec} we notice that
	\begin{align*}
		\int_{\{|x|>2R\}}\big[A(x,\nabla u_n)&+V(x)|u_n|^{\alpha(x)}\big]\diff x\\
		&\leq \int_{ \mathbb{R}^N}\big[A(x,\nabla u_n)+V(x)|u_n|^{\alpha(x)}\big]\phi_R\diff x\\
		&\quad\quad\quad\leq \int_{\{|x|>R\}}\left[A(x,\nabla u_n)+V(x)|u_n|^{\alpha(x)}\right]\diff x.
	\end{align*}
	Thus we obtain
	\begin{equation}\label{P_CCP2.mu_inf}
		\mu_\infty=\lim_{R\to\infty}\underset{n\to\infty}{\lim\sup}\int_{\mathbb{R}^N}\left[A(x,\nabla u_n)+V(x)|u_n|^{\alpha(x)}\right]\phi_R\diff x.
	\end{equation}
	For the second term of the right-hand side of \eqref{P_CCP2.dec}, we notice that $1-\phi_R$ is a continuous function with compact support in $\mathbb{R}^N$. Hence,
	\begin{equation}\label{P_CCP2.phiR}
		\lim_{n\to\infty}\int_{\mathbb{R}^N}\left[A(x,\nabla u_n)+V(x)|u_n|^{\alpha(x)}\right](1-\phi_R)\diff x=\int_{\mathbb{R}^N}(1-\phi_R)\diff\mu.
	\end{equation}
	Clearly, $\lim_{R\to\infty}\int_{\mathbb{R}^N}\phi_R\diff\mu=0$ in view of the Lebesgue dominated convergence theorem. By this and \eqref{P_CCP2.phiR}, we deduce
	\begin{equation*}
		\lim_{R\to\infty}\lim_{n\to \infty }\int_{\mathbb{R}^N}\left[A(x,\nabla u_n)+V(x)|u_n|^{\alpha(x)}\right](1-\phi_R)\diff x=\mu(\mathbb{R}^N).
	\end{equation*}
	Using this and \eqref{P_CCP2.mu_inf}, we obtain \eqref{CCP2.mu} by taking the limit superior as $n\to\infty$ and then letting $R\to\infty$ in \eqref{P_CCP2.dec}. In the same fashion, to obtain \eqref{CCP2.nu} we decompose
	\begin{equation}\label{P_CCP2.dec2}
		\int_{ \mathbb{R}^N}b(x)|u_n|^{t(x)}\diff x=\int_{ \mathbb{R}^N}b(x)|u_n|^{t(x)}\phi_R\diff x+\int_{ \mathbb{R}^N}b(x)|u_n|^{t(x)}(1-\phi_R)\diff x.
	\end{equation}
	Arguing as above, we obtain
	\begin{equation}\label{P_CCP2.nu_inf}
		\nu_\infty=\lim_{R\to\infty}\underset{n\to\infty}{\lim\sup}\int_{\mathbb{R}^N}b(x)|u_n|^{t(x)}\phi_R\diff x,
	\end{equation}
	and using \eqref{P_CCP2.nu_inf}, we easily obtain \eqref{CCP2.nu} from \eqref{P_CCP2.dec2}. Moreover, by replacing $\phi_R$ with $\phi_R^{t(x)}$ in the above aguments we also have
	\begin{equation}\label{P_CCP2.nu_inf'}
		\nu_\infty=\lim_{R\to\infty}\underset{n\to\infty}{\lim\sup}\int_{\mathbb{R}^N}b(x)|u_n|^{t(x)}\phi_R^{t(x)}\diff x.
	\end{equation}

	Next, we prove \eqref{CCP2.nu_mu} when $(\mathcal{E}_\infty)$ is additionally assumed.  It is easy to see that $\phi_Rv\in X$ for all $R>0$ and all $v\in X$. Thus by \eqref{Sb}, we have
	\begin{multline}\label{P_CCP2.est1}
		 S_b|\phi_R u_n|_{L^{t(\cdot)}(b,\R^N)} \leq \big||\nabla(\phi_R u_n)|\big|_{L^{p(\cdot)}(\R^N)+L^{q(\cdot)}(\R^N)}+|\phi_R u_n|_{L^{\alpha(\cdot)}(V,\R^N)}\\
		\leq \big||\phi_R\nabla u_n|\big|_{L^{p(\cdot)}(\R^N)+L^{q(\cdot)}(\R^N)}+|\phi_R u_n|_{L^{\alpha(\cdot)}(V,\R^N)}+\big||u_n\nabla\phi_R|\big|_{L^{p(\cdot)}(\R^N)+L^{q(\cdot)}(\R^N)}.
	\end{multline}
Let $\epsilon$ be arbitrary in $(0,1).$ By $(\mathcal{E}_\infty)$, there exists $R_0=R_0(\epsilon)>0$ such that
	\begin{equation}\label{P_CCP2.p,t_inf}
		|p(x)-p_\infty|<\epsilon,\ |q(x)-q_\infty|<\epsilon,\ |\alpha(x)-\alpha_\infty|<\epsilon,\ |t(x)-t_\infty|<\epsilon,\quad \forall |x|>R_0.
	\end{equation}
	For $R>R_0$ given, let $\{u_{n_k}\}_{k\in\N}$ be a subsequence of $\{u_n\}_{n\in\N}$ such that
	\begin{equation}\label{P_CCP2.limsup}
		\lim_{k\to \infty }\int_{B_R^c}b(x)|u_{n_k}|^{t(x)}\phi_R^{t(x)}\diff x=\underset{n\to\infty}{\lim\sup}\int_{B_R^c}b(x)|u_n|^{t(x)}\phi_R^{t(x)}\diff x.
	\end{equation}
	Using Proposition~\ref{norm-modular} with \eqref{P_CCP2.p,t_inf}, we have
	\begin{equation*}
		|\phi_R u_{n_k}|_{L^{t(\cdot)}(b,B_{R}^c)}
		\geq\min\left\{\left(\int_{B_R^c}b(x)|u_{n_k}|^{t(x)}\phi_R^{t(x)}\diff x\right)^{\frac{1}{t_\infty+\epsilon}},\left(\int_{B_R^c}b(x)|u_{n_k}|^{t(x)}\phi_R^{t(x)}\diff x\right)^{\frac{1}{t_\infty-\epsilon}}\right\}.
	\end{equation*}
	From this, \eqref{P_CCP2.nu_inf'} and \eqref{P_CCP2.limsup} we deduce
	\begin{equation}\label{P_CCP2.est2}
		\lim_{R\to\infty}\underset{{k}\to\infty}{\lim\sup}|\phi_R u_{n_k}|_{L^{t(\cdot)}(b,B_{R}^c)}\geq \min\biggl\{\nu_\infty^{\frac{1}{t_\infty+\epsilon}} , \nu_\infty^{\frac{1}{t_\infty-\epsilon}} \biggr\}.
	\end{equation}
	On the other hand, arguing as that obtained \eqref{P_CCP.lambda_n-E} again with noticing $\supp (|\nabla u_{n_k}|)\subset B_R^c$ we get that
\begin{align}\label{P_CCP2.Grad1}
	&\big||\phi_R\nabla u_{n_k}|\big|_{L^{p(\cdot)}(B_R^c)+L^{q(\cdot)}(B_R^c)}\nonumber\\
&\qquad\qquad\leq \big||\phi_R\nabla u_{n_k}\chi_{\{|\nabla u_{n_k}|\geq 1\}}|\big|_{L^{p(\cdot)}(B_R^c)}+\big||\phi_R\nabla u_{n_k}\chi_{\{|\nabla u_{n_k}|\leq 1\}}|\big|_{L^{q(\cdot)}(B_R^c)}.
\end{align}
Invoking Proposition~\ref{norm-modular} and taking  into account \eqref{P_CCP2.p,t_inf} again, we have
	\begin{align}\label{P_CCP2.Grad2}
		\notag&\big||\phi_R\nabla u_{n_k}\chi_{\{|\nabla u_{n_k}|\geq 1\}}|\big|_{L^{p(\cdot)}(B_R^c)}\\
		\notag&\leq\max\left\{\left(\int_{B_R^c}|\nabla u_{n_k}|^{p(x)}\chi_{\{|\nabla u_{n_k}|\geq 1\}}\phi_R^{p(x)}\diff x\right)^{\frac{1}{p_\infty+\epsilon}},\left(\int_{B_R^c}|\nabla u_{n_k}|^{p(x)}\chi_{\{|\nabla u_{n_k}|\geq 1\}}\phi_R^{p(x)}\diff x\right)^{\frac{1}{p_\infty-\epsilon}}\right\} \\
	\notag	&\leq\max\left\{\left(\int_{B_R^c}M_1^{-1}A(x,\nabla u_{n_k})\phi_R\diff x\right)^{\frac{1}{p_\infty+\epsilon}},\left(\int_{B_R^c}M_1^{-1}A(x,\nabla u_{n_k})\phi_R\diff x\right)^{\frac{1}{p_\infty-\epsilon}}\right\} \\
	\notag	\notag&\leq\max\left\{M_1^{-\frac{1}{p_\infty+\epsilon}},M_1^{-\frac{1}{p_\infty-\epsilon}}\right\}\max\biggl\{\left(\int_{B_R^c}\left[A(x,\nabla u_{n_k})+V(x)|u_{n_k}|^{\alpha(x)}\right]\phi_R\diff x\right)^{\frac{1}{p_\infty+\epsilon}},\\
		&{}\hspace*{6cm}\left(\int_{B_R^c}\left[A(x,\nabla u_{n_k})+V(x)|u_{n_k}|^{\alpha(x)}\right]\phi_R\diff x\right)^{\frac{1}{p_\infty-\epsilon}}\biggr\}.
	\end{align}
Similarly we have
	\begin{align}\label{P_CCP2.Grad3}
	\notag&\big||\phi_R\nabla u_{n_k}\chi_{\{|\nabla u_{n_k}|\leq 1\}}|\big|_{L^{q(\cdot)}(B_R^c)}\\
\notag	&\leq\max\left\{\left(\int_{B_R^c}|\nabla u_{n_k}|^{q(x)}\phi_R^{q(x)}\chi_{\{|\nabla u_{n_k}|\leq 1\}}\diff x\right)^{\frac{1}{q_\infty+\epsilon}},\left(\int_{B_R^c}|\nabla u_{n_k}|^{q(x)}\phi_R^{q(x)}\chi_{\{|\nabla u_{n_k}|\leq 1\}}\diff x\right)^{\frac{1}{q_\infty-\epsilon}}\right\} \\
\notag	&\leq\max\left\{M_1^{-\frac{1}{q_\infty+\epsilon}},M_1^{-\frac{1}{q_\infty-\epsilon}}\right\}\max\biggl\{\left(\int_{B_R^c}\left[A(x,\nabla u_{n_k})+V(x)|u_{n_k}|^{\alpha(x)}\right]\phi_R\diff x\right)^{\frac{1}{q_\infty+\epsilon}},\\
	&{}\hspace*{6cm}\left(\int_{B_R^c}\left[A(x,\nabla u_{n_k})+V(x)|u_{n_k}|^{\alpha(x)}\right]\phi_R\diff x\right)^{\frac{1}{q_\infty-\epsilon}}\biggr\}
\end{align}
and
\begin{align}\label{P_CCP2.Grad4}
	\notag \big|\phi_R& u_{n_k}\big|_{L^{\alpha(\cdot)}(V,B_{R}^c)}\\
\notag	&\leq\max\biggl\{\left(\int_{B_R^c}\left[A(x,\nabla u_{n_k})+V(x)|u_{n_k}|^{\alpha(x)}\right]\phi_R\diff x\right)^{\frac{1}{\alpha_\infty+\epsilon}},\\
	&{}\hspace*{5cm}\left(\int_{B_R^c}\left[A(x,\nabla u_{n_k})+V(x)|u_{n_k}|^{\alpha(x)}\right]\phi_R\diff x\right)^{\frac{1}{\alpha_\infty-\epsilon}}\biggr\}.
\end{align}
From \eqref{P_CCP2.Grad1}--\eqref{P_CCP2.Grad4} and \eqref{P_CCP2.mu_inf} we obtain
\begin{multline}\label{P_CCP2.est3}
\lim_{R\to\infty}\underset{{k}\to\infty}{\lim\sup}\left[\big||\phi_R\nabla u_{n_k}|\big|_{L^{p(\cdot)}(B_R^c)+L^{q(\cdot)}(B_R^c)}+\big|\phi_R u_{n_k}\big|_{L^{\alpha(\cdot)}(V,B_{R}^c)}\right]\\
\leq\max\left\{1, M_1^{-\frac{1}{p_\infty-\epsilon}}\right\} \max\biggl\{\mu_\infty^{\frac{1}{p_\infty+\epsilon}}+\mu_\infty^{\frac{1}{q_\infty+\epsilon}}+ \mu_\infty^{\frac{1}{\alpha_\infty+\epsilon}}, \mu_\infty^{\frac{1}{p_\infty-\epsilon}}+\mu_\infty^{\frac{1}{q_\infty-\epsilon}}+\mu_\infty^{\frac{1}{\alpha_\infty-\epsilon}} \biggr\}.
\end{multline}
Now we analyze the last term in \eqref{P_CCP2.est1} to show \eqref{CCP2.nu_mu}. First we have
	\begin{align*}
		\notag\big||u_n\nabla\phi_R|&\big|_{L^{p(\cdot)}(\R^N)+L^{q(\cdot)}(\R^N)}\\
		&\leq \big||u_n\nabla\phi_R\chi_{\{|u_n\nabla\phi_R|\geq 1\}}|\big|_{L^{p(\cdot)}(\R^N)}+\big||u_n\nabla\phi_R\chi_{\{|u_n\nabla\phi_R|\leq 1\}}|\big|_{L^{q(\cdot)}(\R^N)}.
	\end{align*}
That is,
\begin{multline}\label{P_CCP2.est4}
\big||u_n\nabla\phi_R|\big|_{L^{p(\cdot)}(\R^N)+L^{q(\cdot)}(\R^N)}\\
\leq \big||u_n\nabla\phi_R\chi_{\{|u_n\nabla\phi_R|\geq 1\}}|\big|_{L^{p(\cdot)}(B_{2R}\setminus \overline{B_{R}})}+\big||u_n\nabla\phi_R\chi_{\{|u_n\nabla\phi_R|\leq 1\}}|\big|_{L^{q(\cdot)}(B_{2R}\setminus \overline{B_{R}})}.
\end{multline}
By Proposition~\ref{W}, we have that $X\hookrightarrow\hookrightarrow L^{p(\cdot)}(B_{2R}\setminus \overline{B_R})$ and $X\hookrightarrow\hookrightarrow L^{q(\cdot)}(B_{2R}\setminus \overline{B_R}).$ Hence, one has
	\begin{equation}\label{P_CCP2.est4'}
		\underset{{n}\to\infty}{\lim\sup}\big||u_n\nabla\phi_R\chi_{\{|u_n\nabla\phi_R|\geq 1\}}|\big|_{L^{p(\cdot)}(B_{2R}\setminus \overline{B_{R}})}=\big||u\nabla\phi_R\chi_{\{|u\nabla\phi_R|\geq 1\}}|\big|_{L^{p(\cdot)}(B_{2R}\setminus \overline{B_{R}})}
	\end{equation}
and
\begin{equation}\label{P_CCP2.est4''}
	\underset{{n}\to\infty}{\lim\sup}\big||u_n\nabla\phi_R\chi_{\{|u_n\nabla\phi_R|\leq 1\}}|\big|_{L^{q(\cdot)}(B_{2R}\setminus \overline{B_{R}})}=\big||u\nabla\phi_R\chi_{\{|u\nabla\phi_R|\leq 1\}}|\big|_{L^{q(\cdot)}(B_{2R}\setminus \overline{B_{R}})}.
\end{equation}
	Also we note that
	\begin{multline}\label{P_CCP2.est5}
	\max\left\{\int_{B_{2R}\setminus \overline{B_{R}}}\big|u\nabla\phi_R\chi_{\{|u\nabla\phi_R|\geq 1\}}\big|^{p(x)}\diff x, 	\int_{B_{2R}\setminus \overline{B_{R}}}\big|u\nabla\phi_R\chi_{\{|u\nabla\phi_R|\leq 1\}}\big|^{q(x)}\diff x\right\}\\
	\leq	\int_{B_{2R}\setminus \overline{B_{R}}}\big|u\nabla\phi_R\big|^{p(x)}\diff x.
	\end{multline}
Arguing as that leads to \eqref{P_CCP.est5} we obtain
	\begin{align}\label{P_CCP2.est5'}
	\notag\lim_{R\to\infty}\int_{B_{2R}\setminus \overline{B_{R}}}\big|u\nabla\phi_R\big|^{p(x)}\diff x=0.
\end{align}
Hence, the estimate \eqref{P_CCP2.est5} infers
	$$\lim_{R\to\infty}\int_{B_{2R}\setminus \overline{B_{R}}}\big|u\nabla\phi_R\chi_{\{|u\nabla\phi_R|\geq 1\}}\big|^{p(x)}\diff x=0$$
	and
		$$\lim_{R\to\infty}\int_{B_{2R}\setminus \overline{B_{R}}}\big|u\nabla\phi_R\chi_{\{|u\nabla\phi_R|\leq 1\}}\big|^{q(x)}\diff x=0.$$
		Equivalently,
	\begin{equation*}
		\lim_{R\to\infty}\big||u\nabla\phi_R\chi_{\{|u\nabla\phi_R|\geq 1\}}|\big|_{L^{p(\cdot)}(B_{2R}\setminus \overline{B_{R}})}=	\lim_{R\to\infty}\big||u\nabla\phi_R\chi_{\{|u\nabla\phi_R|\leq 1\}}|\big|_{L^{q(\cdot)}(B_{2R}\setminus \overline{B_{R}})}=0.
	\end{equation*}
From this \eqref{P_CCP2.est4'}--\eqref{P_CCP2.est4''} we deduce from \eqref{P_CCP2.est4} that
\begin{equation}\label{P_CCP2.est4}	\lim_{R\to\infty}\underset{{n}\to\infty}{\lim\sup}\big||u_n\nabla\phi_R|\big|_{L^{p(\cdot)}(\R^N)+L^{q(\cdot)}(\R^N)}=0.
\end{equation}

Utilizing \eqref{P_CCP2.est1}, \eqref{P_CCP2.est2}, \eqref{P_CCP2.est3} and \eqref{P_CCP2.est4} we arrive at
	\begin{multline*}
		S_b\min\biggl\{\nu_\infty^{\frac{1}{t_\infty+\epsilon}} , \nu_\infty^{\frac{1}{t_\infty-\epsilon}} \biggr\}\\
		\leq \max\left\{1, M_1^{-\frac{1}{p_\infty-\epsilon}}\right\} \max\biggl\{\mu_\infty^{\frac{1}{p_\infty+\epsilon}}+\mu_\infty^{\frac{1}{q_\infty+\epsilon}}+ \mu_\infty^{\frac{1}{\alpha_\infty+\epsilon}}, \mu_\infty^{\frac{1}{p_\infty-\epsilon}}+\mu_\infty^{\frac{1}{q_\infty-\epsilon}}+\mu_\infty^{\frac{1}{\alpha_\infty-\epsilon}} \biggr\}.
	\end{multline*}
	Letting $\epsilon\to 0^+$ in the last inequality, we obtain \eqref{CCP2.nu_mu}.
\end{proof}

\section{application}\label{Application}

\smallskip
As an application of Theorems \ref{T.CCP} and \ref{T.CCP.infinity}, in this section we will obtain the existence of a nontrivial nonnegative solution and infinitely many solutions for problem \eqref{Prob} when the reaction term is of generalized concave-convex type.

 Throughout this section, let $(A1)$--$(A6)$, $(P1)$, $(V1)$, $(V2)$ and $(\mathcal{E}_\infty)$ hold. Furthermore, $r \in C_+^{0,1}(\mathbb{R}^N)$  with  $r(\cdot) \ll p^\ast(\cdot) $ we denote $$A_r:=\cup_{\theta\in E_r}L_+^{\theta(\cdot)}(\mathbb{R}^N)$$
 with $E_r:=\left\{\theta\in C_+(\R^N): \alpha(\cdot)\leq  \theta'(\cdot)r(\cdot)\leq p^*(\cdot) \ \text{in} \ \R^N\right\}$, and when $(V3)$ and $\alpha(\cdot)\leq r(\cdot)$ in $\R^N$ are additionally assumed, we denote $$A_r:=\cup_{\theta\in E_r}L_+^{\theta(\cdot)}(\mathbb{R}^N)\cup L_+^\infty(\R^N).$$
 Then by Propositions~\ref{X-Imbeddings} and \ref{IV.compact} we have that
	\begin{equation}\label{IV.CompactImb}
		X \hookrightarrow \hookrightarrow L^{r(\cdot) }(\varrho,\mathbb{R}^N ) \ \ \text{if}\ \ \varrho\in A_r.
	\end{equation}
Assume in addition that
\begin{itemize}
		\item[$(P2)$] $q_\alpha^+<s^+<t^-$, where $s$ and $q_\alpha$ are given in $(A5)$ and \eqref{Def.p_alpha}, respectively.
		\item[$(\mathcal{M})$] $M: [0, \infty) \to \R$ is a real function such that $M$ is continuous and nondecreasing on an interval $[0,\tau_0]$ for some $\tau_0 >0$ and $M(0)>0$.
		\item [$(\mathcal{F}_{1})$] There exist $r_i \in C_+^{0,1}(\mathbb{R}^N)$  with  $r_i(\cdot) \ll p^\ast(\cdot) $ and $\varrho_i\in A_{r_i}$ ($i=1,\cdots,m$) such that
		\begin{equation*}
			|f(x,\tau)|\le \sum_{i=1}^{m} \varrho_i(x)|\tau|^{r_i(x) - 1}\quad \text{for a.e. }x\in \mathbb{R}^N \text{ and all } \tau\in \mathbb{R}.
		\end{equation*}
		
		\item [$(\mathcal{F}_{2})$] There exist $r \in C_+^{0,1}(\mathbb{R}^N)$ with $r^{+}<p_\alpha^{-}$,  $\varrho\in A_r$, $\beta\in (s^+,t^-)$, and $C_i\  (i=1,\cdots,6$) such that
		\begin{equation*}
			C_1\varrho(x)|\tau|^{r(x)}-C_2b(x)|\tau|^{t(x)}\le \beta F(x,\tau) \le f(x,\tau)\tau + C_3\varrho(x)|\tau|^{r(x)}+C_4b(x)|\tau|^{t(x)}
		\end{equation*}
		and
		\begin{equation*}
			F(x,\tau)\le C_5\varrho(x)|\tau|^{r(x)}+C_6b(x)|\tau|^{t(x)}
		\end{equation*}
		for a.e. $x\in \R^N$ and all $\tau \in \mathbb{R}$, where $F(x,\tau):= \int_{0}^{\tau}f(x,s)\diff s$ and $p_\alpha$ is given by \eqref{Def.p_alpha}.
		\end{itemize}

	\smallskip
	A typical example for $f$ fulfilling $(\mathcal{F}_1)$--$(\mathcal{F}_2)$ is a $p(\cdot)$-sublinear term $f(x,t)=\varrho(x)|t|^{r(x)-2}t$ with $r \in C_+^{0,1}(\mathbb{R}^N)$, $r^{+}<p_\alpha^{-}$, and  $\varrho\in A_r$. Furthermore,  if $(V3)$ is additionally assumed, then an interesting another example is $f(x,t)= c_1|t|^{r(x)-2}t+c_2b(x)|t|^{m(x)-2}t\log^\kappa (e+|t|)$ with $c_1>0$, $c_2\geq 0$, $\kappa\geq 0$ and $r,m\in C_+^{0,1}(\mathbb{R}^N)$ satisfying $r^+<p_\alpha^-$ and $p_\alpha(\cdot)\leq m(\cdot)\ll t(\cdot)$.
	
		\smallskip
By a solution of problem~\eqref{Prob}, we mean a function $u\in X$ such that
\begin{multline}
\label{Def.w-sol}
M\left(\int_{\R^N}A(x,\nabla u)\,\diff x\right)\int_{\R^N}a(x,\nabla u)\cdot\nabla v\,\diff x+\int_{\R^N}V(x)|u|^{\alpha(x)-2}u v\diff x\\-\lambda\int_{\R^N} f(x,u)v\,\diff x-\int_{\R^N} b(x)|u|^{t(x)-2}uv\,\diff x=0,\quad \forall v\in X.
\end{multline}
This definition is clearly well defined under above assumptions thanks to the aforementioned imbeddings on $X$.

Our first existence result is the next theorem.
\begin{theorem}[A nontrivial nonnegative solutions for the generalized concave-convex type problem]\label{Theo.Sub1}
	Let $(A1)$--$(A6)$, $(P1)$, $(P2)$, $(V1)$, $(V2)$ and $(\mathcal{E}_\infty)$ hold.  Let $(\mathcal{M})$, $(\mathcal{F}_1)$ and $(\mathcal{F}_2)$ hold with  $\varrho^{\frac{t}{t-r}}b^{-\frac{r}{t-r}}\in L^1(\mathbb{R}^N).$  Then, there exists $\la_{\star}>0$ such that for any $\la\in(0,\la_{\star})$, problem \eqref{Prob} admits a nontrivial nonnegative solution $u_\lambda$. Furthermore, it holds that
	\[
	\lim_{\la\to0^{+}}\|u_{\la}\|=0.
	\]
\end{theorem}
In addition, if $f$ is symmetric with respect to the second variable, then we can obtain infinitely many solutions for problem \eqref{Prob} as follows.
\begin{theorem}[Infinitely many solutions for the generalized concave-convex type problem]\label{Theo.Sub}
	In addition to the assumptions of Theorem~\ref{Theo.Sub1}, assume that $f(x,-\tau)=-f(x,\tau)$ and $F(x,\tau)\geq 0$ for a.e. $x\in \R^N$ and all $\tau \in \mathbb{R}$. Then, there exists $\la_{*}>0$ such that for any $\la\in(0,\la_{*})$, problem \eqref{Prob} admits infinitely many solutions. Furthermore, if let $u_{\la}$ be one of these solutions, then it holds that
	\[
	\lim_{\la\to0^{+}}\|u_\la\|=0.
	\]
\end{theorem}
We have the following important remark.
\begin{remark}\rm
	In Theorems~\ref{Theo.Sub1} and \ref{Theo.Sub}, if we additionally assume $(V3)$ and $\alpha(\cdot)\leq r(\cdot)$, then the weights $\rho,\rho_i$ in $(\mathcal{F}_1)$ and $(\mathcal{F}_2)$ can be $L^\infty$-weights.
\end{remark}

We will make use of critical points theory to determine solutions to problem~\eqref{Prob}.  In order to get necessary properties regarding the Kirchhoff term, we truncate the function $M(t)$ as follows. Let us fix $\bar{\tau}_0\in(0,\tau_0)$ such that
\begin{equation}\label{4.1-t0}
	M(\bar{\tau}_0)<\frac{\beta}{s^+}M(0).
\end{equation}
 Define
\begin{equation}\label{Def.M0}
	M_0(t):=\begin{cases}
		M(\tau),\ \ &0\leq \tau\leq \bar{\tau}_0,\\
		M(\bar{\tau}_0),\ \ &\tau>\bar{\tau}_0
	\end{cases}
\end{equation}
and
$$\widehat{M}_0(\tau):=\int_0^\tau M_0(s)\diff s,\ \ \tau\geq 0.$$
It is clear that $M_0\in C([0,\infty),\R^+)$,
\begin{equation}\label{Est.M0}
	0<m_0:=M(0)\leq M_0(\tau)\leq M(\bar{\tau}_0),\quad \forall \tau\in[0,\infty)
\end{equation}
and
\begin{equation}\label{Est.hat.M0}
	m_0\tau\leq \widehat{M}_0(\tau)\leq M(\bar{\tau}_0)\tau,\quad \forall \tau\in[0,\infty).
\end{equation}
Let $\lambda>0$. We define modified energy functionals $J_\lambda, \widetilde{J}_\lambda: \, X\to\R$ as
\begin{multline*}
	 J_\lambda(u):=\widehat{M}_0\left(\int_{\R^N}A(x,\nabla u)\diff x\right)+\int_{\R^N}\frac{V(x)}{\alpha(x)}|u|^{\alpha(x)}\diff x\\
	-\lambda\int_{\R^N} F(x,u)\diff x-\int_{\R^N}\frac{b(x)}{t(x)}|u|^{t(x)}\diff x,\quad u\in X,
\end{multline*}
and
\begin{multline*}
	\widetilde{J}_\lambda(u):=\widehat{M}_0\left(\int_{\R^N}A(x,\nabla u)\diff x\right)+\int_{\R^N}\frac{V(x)}{\alpha(x)}|u|^{\alpha(x)}\diff x\\
	-\lambda\int_{\R^N} F^+(x,u)\diff x-\int_{\R^N}\frac{b(x)}{t(x)}(u^+)^{t(x)}\diff x,\quad u\in X,
\end{multline*}
where $u^+:=\max\{u,0\}$ and
\begin{equation}\label{positive_F}
	F^+(x,\tau):=\int_0^\tau f^+(x,\eta)\diff\eta \ \ \text{and}\ \ f^{+}(x,\tau):=\begin{cases}
		f(x,\tau),\ \ &\tau\geq 0,\\
		0,\ \ &\tau<0
	\end{cases}
\end{equation}
for a.e. $x\in\R^N$ and all $\tau\in\R$. By a standard argument, we can show that $J_\lambda,\widetilde{J}_\lambda\in C^1(X,\R)$ and its Fr\'echet derivative $J_\lambda',\widetilde{J}_\lambda': X\to X^\ast$ are given by
\begin{multline*}
	\left\langle J_\lambda'(u),v\right\rangle=M_0\left(\int_{\R^N}A(x,\nabla u)\,\diff x\right)\int_{\R^N}a(x,\nabla u)\cdot\nabla v\,\diff x+\int_{\R^N}V(x)|u|^{\alpha(x)-2}u v\diff x\\
	-\lambda\int_{\R^N} f(x,u)v\,\diff x-\int_{\R^N}b(x) |u|^{t(x)-2}uv\,\diff x,\quad \forall\, u,v\in X,
\end{multline*}
and
\begin{multline*}
	\left\langle \widetilde{J}_\lambda'(u),v\right\rangle=M_0\left(\int_{\R^N}A(x,\nabla u)\,\diff x\right)\int_{\R^N}a(x,\nabla u)\cdot\nabla v\,\diff x+\int_{\R^N}V(x)|u|^{\alpha(x)-2}u v\diff x\\
	-\lambda\int_{\R^N} f^+(x,u)v\,\diff x-\int_{\R^N} b(x)(u^+)^{t(x)-1}v\,\diff x,\quad \forall\, u,v\in X.
\end{multline*}
From the definition \eqref{Def.M0} of $M_0$, it is clear that any critical point $u$ of $J_\lambda$ (resp. $\widetilde{J}_\lambda$) is a solution (resp. a nonnegative solution) to problem~\eqref{Prob} provided $\int_{\R^N}A(x,\nabla u)\,\diff x\leq \bar{\tau}_0$.

In the next two subsections, we always assume that $(A1)-(A6)$, $(P1)$, $(P2)$, $(V1)$, $(V2)$, $(\mathcal{E}_\infty)$, $(\mathcal{M})$, $(\mathcal{F}_1)$ and $(\mathcal{F}_2)$ hold with $\varrho^{\frac{t}{t-r}}b^{-\frac{r}{t-r}}\in L^1(\mathbb{R}^N).$ We also make use of the following estimate that is easily derived from \eqref{Sb} and \eqref{IV.CompactImb}:
\begin{equation}\label{4.EstNorms}
	\max\left\{|u|_{L^{r(\cdot)}(\varrho,\R^N)}, |u|_{L^{t(\cdot)}(b,\R^N)} \right\}\leq C_7\|u\|,\ \  \forall u\in X.
\end{equation}
Here and in the rest of this section, $C_i$ ($i=7,8,\cdots$) stand for positive constants depending only on the data, and we can take $C_7>1$.

\medskip
\subsection{\textbf{Existence of a nontrivial nonnegative solution}} ${}$

\smallskip
In this subsection, we will prove Theorem~\ref{Theo.Sub1} via employing the Ekeland variational principle for $\widetilde{J}_\lambda$. For this purpose, we first obtain several auxiliary results. The next lemma provides a certain range of levels such that the local Palas-Smale condition for $\widetilde{J}_\lambda$ is satisfied. In the following, by a $(\textup{PS})_c$-sequence $\{u_n\}_{n\in\mathbb{N}}$ for a $C^1$ functional $I:X\to \R$ we mean \begin{equation*}
	I(u_n)\to c \quad\text{and}\quad I'(u_n)\to 0 \quad\text{as}\quad n\to\infty,
\end{equation*}
and we say that $I$ satisfies the $(\textup{PS})_c$ condition if every $(\textup{PS})_c$-sequence for $I$ admits a convergent subsequence. Set
\begin{equation}\label{lambda1}
	 \lambda_1:=\frac{1}{2}\left(\frac{1}{\beta}-\frac{1}{t^{-}}\right)\beta C_4^{-1},
\end{equation}
where $\beta$ and $C_4$ are given in $(\mathcal{F}_2)$.
\begin{lemma}\label{Le4.3}
For any given $\lambda\in (0,\lambda_1)$, $\widetilde{J}_\lambda$ satisfies the $(\textup{PS})_c$ condition with $c\in\mathbb{R}$ satisfying
\begin{align}\label{c-sub}
c<K_1
-K_2\max\left\{\lambda^{\frac{l^{+}}{l^{+}-1}},\lambda^{\frac{l^{-}}{l^{-}-1}}\right\},
\end{align}
where $l(\cdot):=\frac{t(\cdot)}{r(\cdot)}$, and $K_1,K_2$ are positive constants depending only on the data. The conclusion remains valid if  $\widetilde{J}_\lambda$ is replaced with  $J_\lambda$.
\end{lemma}
\begin{proof}
Let $\lambda\in (0,\lambda_1)$ and let $\{u_n\}_{n\in\mathbb{N}}$ be a $(\textup{PS})_c$-sequence for $\widetilde{J}_\lambda$ in $X$, namely
\begin{equation}\label{PSc}
	\widetilde{J}_\lambda(u_n)\to c \quad\text{and}\quad \widetilde{J}_\lambda'(u_n)\to 0 \quad\text{as}\quad n\to\infty,
\end{equation}
for some $c\in\R$ satisfying \eqref{c-sub}.  From this, \eqref{4.1-t0}, \eqref{Est.M0}--\eqref{positive_F}, $(A5)$ and $(\mathcal{F}_2)$ we have that for $n$ large,
\begin{align}\label{P.le4.3-1}
	\notag c+&1+\|u_n\|\\
\notag	\geq&\widetilde{J}_\lambda(u_n)-\frac{1}{\beta}\left\langle \widetilde{J}_\lambda'(u_n),u_n \right\rangle\\
	\notag\ge&\widehat{M}_0\left(\int_{\mathbb{R}^N}A(x,\nabla u_n)\diff x\right)  -\frac{1}{\beta}M_0\left(\int_{\mathbb{R}^N}A(x,\nabla u_n)\diff x\right) \int_{\mathbb{R}^N}a(x,\nabla u_n)\cdot\nabla u_n \diff x\\
	\notag&+ \left(\frac{1}{\alpha^{+}}-\frac{1}{\beta}\right)\int_{\R^N}V(x)|u_n|^{\alpha(x)} \diff x -\frac{\la}{\beta}\int_{\R^N}\left[\beta F^+(x,u_n)-f^+(x,u_n)u_n\right] \diff x\\
	\notag& +\left(\frac{1}{\beta}-\frac{1}{t^{-}}\right)\int_{\R^N}b(x)(u_n^+)^{t(x)} \diff x\\
	\ge &\left(m_0-\frac{s^+M(\bar{\tau}_0)}{\beta}\right)\int_{\mathbb{R}^N}A(x,\nabla u_n) \diff x + \left(\frac{1}{\alpha^{+}}-\frac{1}{\beta}\right)\int_{\R^N}V(x)|u_n|^{\alpha(x)} \diff x\notag\\
	\notag&-\frac{\la}{\beta}\int_{\R^N}\left[C_3\varrho(x)(u_n^+)^{r(x)}+C_4b(x)(u_n^+)^{t(x)}\right] \diff x +\left(\frac{1}{\beta}-\frac{1}{t^{-}}\right)\int_{\R^N}b(x)(u_n^+)^{t(x)} \diff x\\
	\ge &m_1\mathcal{A}(u_n)-\frac{\la C_3}{\beta}\int_{\R^N}\varrho(x)|u_n|^{r(x)} \diff x+\left[ \left(\frac{1}{\beta}-\frac{1}{t^{-}}\right)-\frac{\la C_4}{\beta}\right]\int_{\R^N}b(x)(u_n^+)^{t(x)} \diff x,
\end{align}
where $m_1:=\min\left\{m_0-\frac{s^+M(\bar{\tau}_0)}{\beta},\frac{1}{\alpha^{+}}-\frac{1}{\beta}\right\}>0$, $\mathcal{A}$ is defined by \eqref{Def.A} and $\bar{\tau}_0$ is given in \eqref{4.1-t0}. By taking into account \eqref{Est.A}, \eqref{4.EstNorms},  Proposition~\ref{norm-modular} and \eqref{lambda1} we derive from \eqref{P.le4.3-1} that for $n$ large,
\begin{align}
	\notag c+1+\|u_n\|
	\ge m_1\left(\|u_n\|^{p_\alpha^-}-1\right)-\frac{\la C_3}{\beta}\left(C_7^{r^+}\|u_n\|^{r^+}+1\right).
\end{align}
Since $p_\alpha^->r^+>1$, the last inequality yields the boundedness of $\{u_n\}_{n\in\mathbb{N}}$; hence,  $\{u_n^+\}_{n\in\mathbb{N}}$ is also a bounded sequence in $X$. Next, we will show that $\{u_n^+\}_{n\in\mathbb{N}}$ is also a $(\textup{PS})_c$-sequence for $\widetilde{J}_\lambda$ in $X$. To this end, using \eqref{a-odd}, \eqref{Est1}, \eqref{Est.A} and \eqref{Est.M0} we have
\begin{align*}
	o_n(1)&=\left\langle \widetilde{J}_\lambda'(u_n),u_n^-\right\rangle\\
	&=M_0\left(\int_{\R^N}A(x,\nabla u_n)\,\diff x\right)\int_{\R^N}a(x,\nabla u_n)\cdot\nabla u_n^-\,\diff x+\int_{\R^N}V(x)|u_n|^{\alpha(x)-2}u_n u_n^-\diff x\\
	&=M_0\left(\int_{\R^N}A(x,\nabla u_n)\,\diff x\right)\int_{\R^N}a(x,\nabla u_n^-)\cdot\nabla u_n^-\,\diff x+\int_{\R^N}V(x)|u_n^-|^{\alpha(x)}\diff x\\
	&\geq m_0\int_{\R^N}A(x,\nabla u_n^-)\,\diff x+\int_{\R^N}V(x)|u_n^-|^{\alpha(x)}\diff x\\
	&\geq \widetilde{m}_0\mathcal{A}(u_n^-)\\
	&\geq \widetilde{m}_0 \alpha_1\min\left\{\|u_n^-\|^{p_\alpha^-},\|u_n^-\|^{q_\alpha^+}\right\},
\end{align*}
where $u_n^-:=\min\{u_n,0\}$ and $\widetilde{m}_0:=\min\{1,m_0\}$. It follows that
\begin{equation}\label{lim-neg-un}
	\lim_{n\to\infty}\mathcal{A}(u_n^-)=\lim_{n\to\infty}\|u_n^-\|=0.
\end{equation}
It is not difficult to see that
\begin{equation}\label{neg-un1}
	\widetilde{J}_\lambda(u_n)=\widetilde{J}_\lambda(u_n^+)+o_n(1)
	\end{equation}
and
\begin{equation}\label{neg-un2}
	\|\widetilde{J}_\lambda'(u_n^+)\|_{X^*}\leq C\left[\|\widetilde{J}_\lambda'(u_n)\|_{X^*}+\mathcal{A}(u_n^-)^{\frac{1}{p_\alpha^-}}+\mathcal{A}(u_n^-)^{\frac{1}{q_\alpha^+}}\right]
\end{equation}
for some positive constant $C$ independent of $n$. From \eqref{PSc} and \eqref{lim-neg-un}--\eqref{neg-un2}, we have that $\{u_n^+\}_{n\in\mathbb{N}}$ is a $(\textup{PS})_c$-sequence for $\widetilde{J}_\lambda$ in $X$.

Set $v_n:=u_n^+$ for $n\in\N$. Then, $v_n\geq 0$ a.e. in $\R^N$ and $\{v_n\}_{n\in\mathbb{N}}$ is a bounded $(\textup{PS})_c$-sequence for $\widetilde{J}_\lambda$ in $X$. By the reflexivity of $X$ and Theorems \ref{T.CCP}--\ref{T.CCP.infinity}, we find $\{x_i\}_{i\in\mathcal{I}}\subset\mathcal{C}$ with $\mathcal{I}$ at most countable such that, up to a subsequence, we have
\begin{align}
v_n(x)&\to u(x)\geq 0\quad\text{a.e. }x\in\Bbb R^{N},\label{31}\\
v_n &\rightharpoonup u\quad\text{in }X,\label{4.w-conv}\\
A(x,\nabla v_n)+V(x)v_n^{\alpha(x)} &\overset{\ast }{\rightharpoonup }\mu\geq A(x,\nabla u) + V(x)u^{\alpha(x)}+\sum_{i\in \mathcal{I}} \mu_i \delta_{x_i}\quad \text{in}\quad \mathcal{M}(\R^N),\label{33}\\
bv_n^{t(x)}&\overset{\ast }{\rightharpoonup }\nu=bu^{t(x)} + \sum_{i\in \mathcal{I}}\nu_i\delta_{x_i}\quad\text{in}\quad\mathcal{M}(\R^N),\label{34}\\
S_b \nu_i^{\frac{1}{t(x_i)}} &\leq 2\max\left\{ M_1^{-\frac{1}{p(x_i)}},M_1^{-\frac{1}{q(x_i)}}\right\} \max\left\{\mu_i^{\frac{1}{p(x_i)}},\mu_i^{\frac{1}{q(x_i)}}\right\}, \  \ \forall i\in \mathcal{I} \label{35}
\end{align}
and
\begin{eqnarray}
\underset{n\to\infty}{\lim\sup}\int_{\mathbb{R}^N}\left[A(x,\nabla v_n)+V(x)v_n^{\alpha(x)}\right]\diff x=\mu(\mathbb{R}^N)+\mu_\infty,\label{4..mu}\\
\underset{n\to\infty}{\lim\sup}\int_{\mathbb{R}^N}b(x)v_n^{t(x)}\diff x=\nu(\mathbb{R}^N)+\nu_\infty,\label{4.nu}
\end{eqnarray}
\begin{equation}\label{CCP2.nu_mu2}
S_b\nu_\infty^{\frac{1}{t_\infty}}\\
\leq \max\left\{1, M_1^{-\frac{1}{p_\infty}}\right\} \left(\mu_\infty^{\frac{1}{p_\infty}}+\mu_\infty^{\frac{1}{q_\infty}}+ \mu_\infty^{\frac{1}{\alpha_\infty}}\right).
\end{equation}
We claim that $\mathcal{I}=\emptyset$ and $\mu_\infty=\nu_\infty=0$. To this end, let us suppose on the contrary that this does not hold. We first consider the case that there exists $i\in \mathcal{I}$. Let $\epsilon>0$ and define $\phi_{i,\epsilon}$ as in the proof of Theorem \ref{T.CCP}. By \eqref{Est1} and \eqref{Est.M0} we have
\begin{align*}
\widetilde{m}_0&\int_{\R^N}\phi_{i,\epsilon}\left[A(x,\nabla v_n)+V(x)v_n^{\alpha(x)}\right]\diff x-\int_{\R^N}\phi_{i,\epsilon}b(x)v_n^{t(x)}\diff x\\
\le& M_0\left(\int_{\R^N}A(x,\nabla v_n)\,\diff x\right)\int_{\R^N}\phi_{i,\epsilon}a(x,\nabla v_n)\cdot\nabla v_n\diff x+\int_{\R^N}\phi_{i,\epsilon}V(x)v_n^{\alpha(x)}\diff x\\
&-\int_{\R^N}\phi_{i,\epsilon}b(x)v_n^{t(x)}\diff x\\
=&\left\langle \widetilde{J}_\lambda'(v_n),\phi_{i,\epsilon}v_n \right\rangle - M_0\left(\int_{\R^N}A(x,\nabla v_n)\,\diff x\right)\int_{\R^N}a(x,\nabla v_n)v_n\cdot\nabla\phi_{i,\epsilon}\diff x\\
& +\la\int_{\R^N}\phi_{i,\epsilon}f(x,v_n)v_n \diff x.
\end{align*}
 From this and \eqref{Est.M0}, we have
\begin{multline}\label{(4.12)}
	\widetilde{m}_0\int_{\R^N}\phi_{i,\epsilon}\left[A(x,\nabla v_{n})+V(x)|v_{n}|^{\alpha(x)}\right]\diff x-\int_{\R^N}\phi_{i,\epsilon}b(x)v_n^{t(x)}\diff x\\
	\le \left| \left\langle \widetilde{J}_\lambda'(v_n),\phi_{i,\epsilon}v_n \right\rangle\right| +\la\int_{\R^N}\phi_{i,\epsilon}|f(x,v_n)v_n| \diff x +M(\bar{\tau}_0)\int_{\R^N}|a(x,\nabla v_n)|| v_n||\nabla\phi_{i,\epsilon}|\diff x.
\end{multline}
We will show that $\limsup_{\epsilon\to0^{+}}\limsup_{n\to\infty} T(v_n,\phi_{i,\epsilon})=0$, where $T(v_n,\phi_{i,\epsilon})$ is each term in the right-hand side of \eqref{(4.12)}. To this end, we first note that by the boundedness of $\{v_n\}$ in $X$, we have
\begin{equation}\label{37}
C_*:=\sup_{n\in\mathbb{N}}\int_{\Bbb R^{N}}\left[|\nabla v_n|^{p(x)}\chi_{\{|\nabla v_n|\ge 1\}}+|\nabla v_n|^{q(x)}\chi_{\{|\nabla v_n|\le 1\}}\right]\diff x<\infty.
\end{equation}
Noticing the boundedness of $\{\phi_{i,\epsilon}v_n\}$ in $X$, it follows from  \eqref{PSc} that
\begin{equation}\label{39}
\limsup_{\epsilon\to0^{+}}\limsup_{n\to\infty}\left| \left\langle \widetilde{J}_\lambda'(v_n),\phi_{i,\epsilon}v_n \right\rangle\right|=0.
\end{equation}
By \eqref{4.w-conv} and $(\mathcal{F}_{1})$, we have $v_n\to u$ in $L^{r_i(\cdot)}(\varrho_i,\R^N)$ ($i=1,\cdots,m$) in view of \eqref{IV.CompactImb}. From this, $(\mathcal{F}_1)$ and \eqref{31}, invoking the Lebesgue dominated convergence theorem, we easily see that
\begin{equation} \label{(4.18)}
\limsup_{\epsilon\to0^{+}}\limsup_{n\to\infty}\int_{\R^N}\phi_{i,\epsilon}|f(x,v_n)v_n| \diff x =\limsup_{\epsilon\to0^{+}}\int_{\R^N}\phi_{i,\epsilon}|f(x,u)u| \diff x =0.
\end{equation}
Finally, invoking the Young inequality, $(A3)$ and \eqref{37} we deduce that for an arbitrary $\delta>0$,
\begin{align}\label{(4.19)}
\notag&\int_{\R^N}|a(x,\nabla v_n)| |v_n||\nabla\phi_{i,\epsilon}|\diff x\\
\notag&\leq\int_{\R^N}\left[\delta|a(x,\nabla v_n)|^{\frac{p(x)}{p(x)-1}} +C(\delta)|v_n|^{p(x)}|\nabla\phi_{i,\epsilon}|^{p(x)}\right]\diff x\\
\notag&\le \widetilde{C}\delta\int_{\R^N}\left[|\nabla v_n|^{p(x)} \chi_{\{|\nabla v_n|\geq 1\}}+|\nabla v_n|^{\frac{p(x)(q(x)-1)}{p(x)-1}} \chi_{\{|\nabla v_n|\leq 1\}}\right]\diff x+C(\delta)\int_{\R^N}|v_n|^{p(x)}|\nabla\phi_{i,\epsilon}|^{p(x)}\diff x\\
\notag&\le \widetilde{C}\delta\int_{\R^N}\left[|\nabla v_n|^{p(x)} \chi_{\{|\nabla v_n|\geq 1\}}+|\nabla v_n|^{q(x)} \chi_{\{|\nabla v_n|\leq 1\}}\right]\diff x+C(\delta)\int_{\R^N}|v_n|^{p(x)}|\nabla\phi_{i,\epsilon}|^{p(x)}\diff x\\
&\le \widetilde{C}C_*\delta+C(\delta)\int_{\R^N}|v_n|^{p(x)}|\nabla\phi_{i,\epsilon}|^{p(x)}\diff x
\end{align}
for a positive constant $\widetilde{C}$. By \eqref{4.w-conv} and Proposition~\ref{W} we have
\begin{equation*}
	\int_{\R^N}|v_n|^{p(x)}|\nabla\phi_{i,\epsilon}|^{p(x)}\diff x=	\int_{B_\epsilon(x_i)}|v_n|^{p(x)}|\nabla\phi_{i,\epsilon}|^{p(x)}\diff x\to \int_{B_\epsilon(x_i)}|u|^{p(x)}|\nabla\phi_{i,\epsilon}|^{p(x)}\diff x
\end{equation*}
as $n\to\infty$. From this and \eqref{(4.19)} we obtain
\begin{equation}\label{(4.17)}
	\limsup_{n\to\infty}\int_{\R^N}|a(x,\nabla v_n)| |v_n||\nabla\phi_{i,\epsilon}|\diff x\leq \widetilde{C}C_*\delta+C(\delta)\int_{B_\epsilon(x_i)}|u|^{p(x)}|\nabla\phi_{i,\epsilon}|^{p(x)}\diff x.
\end{equation}
Arguing as that obtained \eqref{P_CCP.est5} we get
\begin{equation*}
	\limsup_{\epsilon\to0^{+}}\limsup_{n\to\infty}\int_{B_\epsilon(x_i)}\big|u\nabla\phi_{i,\epsilon}\big|^{p(x)}\diff x=0.
\end{equation*}
Thus, \eqref{(4.17)} implies that
\begin{equation*}
\limsup_{\epsilon\to0^{+}}\limsup_{n\to\infty}\int_{\R^N}|a(x,\nabla v_n)| |v_n||\nabla\phi_{i,\epsilon}|\diff x\leq \widetilde{C}C_*\delta.
\end{equation*}
Since $\delta>0$ was taken arbitrarily, the preceding inequality leads to
\begin{equation}\label{44}
	\limsup_{\epsilon\to0^{+}}\limsup_{n\to\infty}\int_{\R^N}|a(x,\nabla v_n)| |v_n||\nabla\phi_{i,\epsilon}|\diff x=0.
\end{equation}
By passing to the limit superior as $n\to\infty$ and taking into account \eqref{39}, \eqref{(4.18)} and \eqref{44}, we deduce from \eqref{(4.12)} that
\begin{equation*}
\widetilde{m}_0\mu_{i}\leq\nu_{i}.
\end{equation*}
Combining this with \eqref{35} gives
\begin{align*}
\mu_i\geq& \min\left\{\left(2^{-1}S_{b}\right)^{\frac{t(x_i)p(x_i)}{t(x_i)-p(x_i)}},\left(2^{-1}S_{b}\right)^{\frac{t(x_i)q(x_i)}{t(x_i)-q(x_i)}}\right\}
\min\left\{M_{1}^{\frac{t(x_i)p(x_i)}{(t(x_i)-p(x_i))q(x_i)}},M_{1}^{\frac{t(x_i)q(x_i)}{(t(x_i)-q(x_i))p(x_{i})}}\right\}\\
&\qquad\times\min\left\{\widetilde{m}_0^{\frac{p(x_i)}{t(x_i)-p(x_i)}},\widetilde{m}_0^{\frac{q(x_i)}{t(x_i)-q(x_i)}}\right\}.
\end{align*}
	Hence, one has
	\begin{equation}\label{4.Est.mu_i-1}
	\nu_i\geq\widetilde{m}_0\mu_i\geq k_1
\end{equation}
with  $h_{p}(x):=\frac{p(x)}{t(x)-p(x)}$, $h_{q}(x):=\frac{q(x)}{t(x)-q(x)}$ for  $x\in\mathbb{R}^N$ and
\begin{equation*}
	k_1:=\min\left\{\left(2^{-1}S_{b}\right)^{\left(h_{p}t\right)^-},\left(2^{-1}S_{b}\right)^{\left(h_{q}t\right)^+}\right\}
	\min\left\{M_1^{\left(\frac{h_{p}t}{q}\right)^-},M_1^{\left(\frac{h_{q}t}{p}\right)^+}\right\}\min\left\{\widetilde{m}_0^{1+h_p^-},\widetilde{m}_0^{1+h_q^+}\right\}.
\end{equation*}

Next, we consider the other case, namely, $\mu_\infty>0$. Let $\phi_R$ be as in the proof of Theorem~\ref{T.CCP.infinity}.  Arguing as that obtained \eqref{(4.12)} we have
\begin{multline}\label{4.23}
\widetilde{m}_0\int_{\R^N}\phi_R\left[A(x,\nabla v_n)+V(x)v_n^{\alpha(x)}\right]\diff x-\int_{\R^N}\phi_Rb(x)v_n^{t(x)}\diff x\\
\le \left| \left\langle \widetilde{J}_\lambda'(v_n),\phi_Rv_n \right\rangle\right| +\la\int_{\R^N}\phi_R|f(x,v_n)v_n| \diff x +M(\bar{\tau}_0)\int_{\R^N}|a(x,\nabla v_n)|| v_n||\nabla\phi_R|\diff x.
\end{multline}
By the same way of deriving \eqref{39}, \eqref{(4.18)} and \eqref{44}, we get
\begin{align*}\label{(4.25)}
\notag	\lim_{R\to\infty}\underset{n\to\infty}{\lim\sup}\left| \left\langle \widetilde{J}_\lambda'(v_n),\phi_Rv_n \right\rangle\right|&=\lim_{R\to\infty}\underset{n\to\infty}{\lim\sup}\int_{\R^N}\phi_R|f(x,v_n)v_n|\diff x \\
	\notag&=\lim_{R\to\infty}\underset{n\to\infty}{\lim\sup}\int_{\R^N}|a(x,\nabla v_n)|| v_n||\nabla\phi_R|\diff x\\
	&=0.
\end{align*}
 Using this, \eqref{P_CCP2.mu_inf} and \eqref{P_CCP2.nu_inf} while  passing to the limit superior as $n\to\infty$ and then $R\to\infty$ in  \eqref{4.23}, we derive
\begin{equation*}
\widetilde{m}_0\mu_{\infty}\leq\nu_{\infty}.
\end{equation*}
Combining this with \eqref{CCP2.nu_mu} gives
\begin{multline*}
	\mu_\infty\geq \min\left\{\left(3^{-1}S_{b}\right)^{\frac{t_\infty p_\infty}{t_\infty-p_\infty}},\left(3^{-1}S_{b}\right)^{\frac{t_\infty q_\infty}{t_\infty-q_\infty}},\left(3^{-1}S_{b}\right)^{\frac{t_\infty \alpha_\infty}{t_\infty-\alpha_\infty}}\right\}\\
	\times\min\left\{1,M_{1}^{\frac{t_\infty}{t_\infty-p_\infty}},M_{1}^{\frac{t_\infty \alpha_\infty}{(t_\infty-\alpha_\infty)p_\infty}},M_{1}^{\frac{t_\infty q_\infty}{(t_\infty-q_\infty)p_\infty}}\right\}\\
	\times
	\min\left\{\widetilde{m}_0^{\frac{p_\infty}{t_\infty-p_\infty}},\widetilde{m}_0^{\frac{q_\infty}{t_\infty-q_\infty}},\widetilde{m}_0^{\frac{\alpha_\infty}{t_\infty-\alpha_\infty}}\right\}.
\end{multline*}
We get from the last two inequalities that
\begin{equation}\label{4.Est.mu_i-2}
		\nu_\infty\geq\widetilde{m}_0\mu_\infty\geq k_2,
\end{equation}
where
\begin{equation*}	k_2:=\min\left\{\left(3^{-1}S_{b}\right)^{\left(h_{p_\alpha}t\right)^-},\left(3^{-1}S_{b}\right)^{\left(h_qt\right)^+}\right\}\min\left\{1,M_{1}^{\left(\frac{h_qt}{p}\right)^+}\right\}\min\left\{\widetilde{m}_0^{1+h_{p_\alpha}^-},\widetilde{m}_0^{1+h_q^+}\right\}
\end{equation*}
and
\begin{equation*}	
h_{p_{\alpha}}(\cdot):=\frac{p_{\alpha}(\cdot)}{t(\cdot)-p_{\alpha}(\cdot)}.
\end{equation*}

From the estimate \eqref{4.Est.mu_i-1} for the case $\mathcal{I}\ne\emptyset$ and the estimate \eqref{4.Est.mu_i-2} for the case $\mu_\infty>0$ we arrive at
\begin{align}\label{4.Est.mu_i-3}
	\sum_{i\in \mathcal{I}\cup \{\infty\}}\nu_i\geq \sum_{i\in \mathcal{I}\cup \{\infty\}}\widetilde{m}_0\mu_i\geq \min\{k_1,k_2\}=:K_0.
\end{align}
Arguing as those lead to \eqref{P.le4.3-1} we have
\begin{align*}\label{4.26}
\notag c+o_n(1)
=&\widetilde{J}_\lambda(v_n)-\frac{1}{\beta}\left\langle \widetilde{J}_\lambda'(v_n),v_n \right\rangle\\
\ge &\left(m_0-\frac{s^+M(\bar{\tau}_0)}{\beta}\right)\int_{\mathbb{R}^N}A(x,\nabla v_n) \diff x + \left(\frac{1}{\alpha^{+}}-\frac{1}{\beta}\right)\int_{\R^N}V(x)|v_n|^{\alpha(x)} \diff x\notag\\
\notag&-\frac{\la}{\beta}\int_{\R^N}\left[C_3\varrho(x)v_n^{r(x)}+C_4b(x)v_n^{t(x)}\right] \diff x +\left(\frac{1}{\beta}-\frac{1}{t^{-}}\right)\int_{\R^N}b(x)v_n^{t(x)} \diff x\\
\ge &\left[ \left(\frac{1}{\beta}-\frac{1}{t^{-}}\right)-\frac{\la C_4}{\beta}\right]\int_{\R^N}b(x)v_n^{t(x)} \diff x -\frac{\la C_3}{\beta}\int_{\R^N}\varrho(x)v_n^{r(x)} \diff x.
\end{align*}
Since $0<\la<\lambda_1= \frac{\left(t^--\beta\right)}{2C_4t^-}$, we deduce from the last estimate that
\begin{align*}
	c+o_n(1)
	\ge &\frac{\left(t^--\beta\right)}{2\beta t^-}\int_{\R^N}b|v_n|^{t(x)} \diff x -\frac{\la C_3}{\beta}\int_{\R^N}\varrho|v_n|^{r(x)} \diff x.
\end{align*}
Passing to the limit as $n\to\infty$ in the last inequality, invoking  \eqref{34}, \eqref{4.nu} and \eqref{4.Est.mu_i-3}, we obtain
\begin{equation}\label{4.28}
c\ge \frac{\left(t^--\beta\right)}{2\beta t^-}\int_{\R^N}b|u|^{t(x)}\diff x - \frac{\la C_3}{\beta}\int_{\R^N}\varrho|u|^{r(x)}\diff x+K_1,
\end{equation}
where
\begin{equation*}
 K_1:=\frac{\left(t^--\beta\right)K_0}{2\beta t^-}.
\end{equation*}
Invoking Proposition \ref{Holder}, we have
\begin{equation*}
\int_{\R^N}\varrho|u|^{r(x)}\diff x = \int_{\R^N}b(\varrho b^{-1})|u|^{r(x)}\diff x \le 2|\varrho b^{-1}|_{L^{\frac{l(\cdot)}{l(\cdot) - 1}}(b, \R^N)}|u^{r(\cdot)}|_{L^{l(\cdot)}(b,\R^N)},
\end{equation*}
where $l(\cdot):=\frac{t(\cdot)}{r(\cdot)}$. Combining the last inequality with \eqref{4.28} gives
\begin{equation}\label{5.18}
c\ge a_{1}\int_{\R^N}b(|u|^{r(x)})^{l(x)}\diff x- b_{1}\lambda \big||u|^{r(\cdot)}\big|_{L^{l(\cdot)}(b,\R^N)} + K_1,
\end{equation}
where $a_{1}:= \frac{t^--\beta}{2\beta t^-}$, $b_{1}:=\frac{2C_3}{\beta}|\varrho b^{-1}|_{L^{\frac{l(\cdot)}{l(\cdot) - 1}}(b,\R^N)}>0$. Note that by Proposition \ref{norm-modular} we have
\begin{equation*}
	\int_{\R^N}b(|u|^{r(x)})^{l(x)}\diff x\geq \min\left\{\big||u|^{r(\cdot)}\big|_{L^{l(\cdot)}(b,\R^N)}^{l^-}, \big||u|^{r(\cdot)}\big|_{L^{l(\cdot)}(b,\R^N)}^{l^+}\right\}.
\end{equation*}
Using this fact, we consider the following two cases.

$\bullet$ If $||u|^{r(\cdot)}|_{L^{l(\cdot)}(b,\R^N)}\ge 1$, then \eqref{5.18} yields
\begin{equation*}
c\ge a_{1} \xi^{l^{-}} - b_{1}\lambda \xi + K_1=:g_{1}(\xi)\quad\text{with }\xi = \big||u|^{r(\cdot)}\big|_{L^{l(\cdot)}(b,\R^N)}\ge 1.
\end{equation*}
Thus,
\begin{equation*}
c\ge \min_{\xi\ge 0} g_{1}(\xi) = g_{1}\left(\left(\frac{b_{1}\lambda}{l^{-}a_{1}}\right)^{\frac{1}{l^{-}-1}}\right),
\end{equation*}
i.e.,
\begin{equation*}
c\ge K_1 - \left[(l^{-})^{-\frac{1}{l^{-}-1}}-(l^{-})^{-\frac{l^{-}}{l^{-}-1}}\right]a_{1}^{-\frac{1}{l^{-}-1}}b_{1}^{\frac{l^{-}}{l^{-}-1}}\lambda^{\frac{l^{-}}{l^{-}-1}}.
\end{equation*}

$\bullet$ If $||u|^{r(\cdot)}|_{L^{l(\cdot)}(b,\R^N)}< 1$, then \eqref{5.18} yields
\begin{equation*}
c\ge a_{1} \xi^{l^{+}} - b_{1}\lambda \xi + K_1=:g_{2}(\xi)\quad\text{with }\xi = \big||u|^{r(\cdot)}\big|_{L^{l(\cdot)}(b,\R^N)}< 1.
\end{equation*}
Thus,
\begin{equation*}
c\ge \min_{\xi\ge 0} g_{2}(\xi) = g_{2}\left(\left(\frac{b_{1}\lambda}{l^{+}a_{1}}\right)^{\frac{1}{l^{+}-1}}\right),
\end{equation*}
i.e.,
\begin{equation*}
c\ge K_1 - \left[(l^{+})^{-\frac{1}{l^{+}-1}}-(l^{+})^{-\frac{l^{+}}{l^{+}-1}}\right]a_{1}^{-\frac{1}{l^{+}-1}}b_{1}^{\frac{l^{+}}{l^{+}-1}}\lambda^{\frac{l^{+}}{l^{+}-1}}.
\end{equation*}
Therefore, in any case, we obtain
\begin{equation*}
c\ge K_1- K_2\max\left\{\lambda^{\frac{l^{+}}{l^{+}-1}},\lambda^{\frac{l^{-}}{l^{-}-1}}\right\},
\end{equation*}
where
$$K_2:= \max_{*\in\{+,-\}}\left[(l^{*})^{-\frac{1}{l^{*}-1}} -(l^{*})^{-\frac{l^{*}}{l^{*}-1}}\right]a_{1}^{-\frac{1}{l^{*}-1}}b_{1}^{\frac{l^{*}}{l^{*}-1}}.$$
This contradicts with \eqref{c-sub}; that is, we have shown that $\mathcal{I} =\emptyset$ and $\mu_\infty=\nu_\infty=0$. Hence, \eqref{34} yields $\int_{\R^N}bv_n^{t(x)}\diff x\to \int_{\R^N}bu^{t(x)}\diff x$.  From this and \eqref{31} we obtain
\begin{equation}\label{4.31}
v_n\to u\quad\text{in }L^{t(\cdot)}(b,\R^N)
\end{equation}
in view of Lemma \ref{V.brezis-lieb}. Moreover, we also have
\begin{equation}\label{4.32}
	v_n\to u \quad\text{in }L^{r_i(\cdot)}(\varrho_i,\R^N)\ (i=1,\cdots,m)
\end{equation}
in view of \eqref{IV.CompactImb} and \eqref{4.w-conv}.
Using \eqref{4.31}, \eqref{4.32}, $(\mathcal{F}_1)$ and invoking the H\"{o}lder type inequality (Proposition \ref{Holder}), we easily obtain
\begin{equation}\label{4.34}
	\int_{\R^N}f(x,v_n)(v_n-u)\diff x\to 0,\ \int_{\R^N}bv_n^{t(x)-1}(v_n-u)\diff x\to 0.
\end{equation}

On the other hand, by the monotonicity of $\xi\mapsto a(x,\xi)$ (due to $(A2)$) and $\tau\mapsto |\tau|^{\alpha(x)-2} \tau$ we have
\begin{align*}
	0\leq&\int_{\R^N}\left[\left(a(x,\nabla v_n)-a(x,\nabla u)\right)\cdot(\nabla v_n-\nabla u)+V(x)\left(|v_n|^{\alpha(x)-2}v_n-|u|^{\alpha(x)-2}u\right)(v_n-u)\right]\diff x\\
	\leq &\widetilde{ m}_0^{-1}M_0\left(\int_{\mathbb{R}^N}A(x,\nabla v_n)\diff x\right)\int_{\R^N}\left(a(x,\nabla v_n)-a(x,\nabla u)\right)\cdot(\nabla v_n-\nabla u)\diff x\\
	&+\widetilde{m}_0^{-1}\int_{\mathbb{R}^N}V(x)\left(|v_n|^{\alpha(x)-2}v_n-|u|^{\alpha(x)-2}u\right)(v_n-u)\diff x\\
		\leq &\widetilde{ m}_0^{-1}\left[M_0\left(\int_{\mathbb{R}^N}A(x,\nabla v_n)\diff x\right)\int_{\R^N}a(x,\nabla v_n)\cdot(\nabla v_n-\nabla u)\diff x+\int_{\mathbb{R}^N}V(x)|v_n|^{\alpha(x)-2}v_n(v_n-u)\diff x\right]\\
	&-\widetilde{ m}_0^{-1}\left[M_0\left(\int_{\mathbb{R}^N}A(x,\nabla v_n)\diff x\right)\int_{\R^N}a(x,\nabla u)\cdot(\nabla v_n-\nabla u)\diff x+\int_{\mathbb{R}^N}V(x)|u|^{\alpha(x)-2}u(v_n-u)\diff x\right].
\end{align*}
This leads to
\begin{align}\label{4.EstS+}
\notag 0\leq&\int_{\R^N}\left[\left(a(x,\nabla v_n)-a(x,\nabla u)\right)\cdot(\nabla v_n-\nabla u)+V(x)\left(|v_n|^{\alpha(x)-2}v_n-|u|^{\alpha(x)-2}u\right)(v_n-u)\right]\diff x\\
\leq & \widetilde{ m}_0^{-1}\left[\left\langle \widetilde{J}_\lambda'(v_n),v_n-u \right\rangle+\int_{\mathbb{R}^N}f(x,v_n)(v_n-u)\diff x+\int_{\mathbb{R}^N}b(x)v_n^{t(x)-1}(v_n-u)\diff x\right]-\widetilde{ m}_0^{-1}H_n,
\end{align}
where
$$H_n:=M_0\left(\int_{\mathbb{R}^N}A(x,\nabla v_n)\diff x\right)\int_{\R^N}a(x,\nabla u)\cdot(\nabla v_n-\nabla u)\diff x+\int_{\mathbb{R}^N}V(x)|u|^{\alpha(x)-2}u(v_n-u)\diff x.$$
By the boundedness of $\{v_n\}_{n\in\mathbb{N}}$ in $X$, \eqref{Est.M0}, \eqref{PSc},  and \eqref{4.w-conv} we easily see that
\begin{equation}\label{4.36}
\lim_{n\to\infty} 	\left\langle \widetilde{J}_\lambda'(v_n),v_n-u \right\rangle=\lim_{n\to\infty} H_n=0.
\end{equation}
Utilizing \eqref{4.34} and \eqref{4.36} we infer from \eqref{4.EstS+} that
\begin{equation*}
	\lim_{n\to\infty}\int_{\R^N}\left[\left(a(x,\nabla v_n)-a(x,\nabla u)\right)\cdot(\nabla v_n-\nabla u)+V(x)\left(|v_n|^{\alpha(x)-2}v_n-|u|^{\alpha(x)-2}u\right)(v_n-u)\right]\diff x=0.
\end{equation*}
Invoking \eqref{4.w-conv} again, we obtain from the last equality that
\begin{equation*}
	\lim_{n\to\infty}\int_{\R^N}\left[a(x,\nabla v_n)\cdot(\nabla v_n-\nabla u)+V(x)|v_n|^{\alpha(x)-2}v_n(v_n-u)\right]\diff x=0.
\end{equation*}
From this and \eqref{4.w-conv} we derive $v_n\to u$ in $X$ in view of \cite[Lemma 4.7-(ii)]{Zhang-Radulescu.2018}. Combining this with \eqref{lim-neg-un} gives $u_n=v_n-u_n^-\to u$. Finally, to obtain the conclusion for $J_\lambda$, we directly argue with $\{u_n\}_{n\in\mathbb{N}}$ instead of $\{v_n\}_{n\in\mathbb{N}}$ in the above arguments. The proof is complete.
\end{proof}
The next two lemmas provide several geometries of $\widetilde{J}_\lambda$. Set
\begin{equation}\label{4.delta0}
	\delta_0:=\min\left\{C_7^{-1}, \left(C_7^{-t^+}4^{-1}\widetilde{ m}_0\alpha_1\right)^{\frac{1}{t^--q_\alpha^+}},1\right\},
\end{equation}
where $\widetilde{ m}_0:=\min\{1,m_0\}$ while $q_\alpha$, $\alpha_1$ and $C_7$ are given by \eqref{Def.p_alpha}, \eqref{Est.A} and \eqref{4.EstNorms}, respectively.
%----------------LEMMA 3.5 ABOUT POSITIVENESS ON A BALL-------------- %
\begin{lemma} \label{Le4.4}
For each $\delta\in (0,\delta_0)$, there exists $\lambda_2=\lambda_2(\delta)$ such that for any $\lambda\in (0,\lambda_2)$, there exists $\rho_\lambda>0$ such that
\begin{equation*}
	\widetilde{J}_\lambda(u)\geq \rho_\lambda,\quad \forall u\in \partial B_\delta,
\end{equation*}
where $\partial B_\delta:=\left\{u\in X: \|u\|=\delta\right\}$.
\end{lemma}
\begin{proof}
	Let $\delta\in (0,\delta_0)$ be given and let $\lambda<(C_6t^-)^{-1}$ with $C_6$ given in $(\mathcal{F}_2)$. For $u\in\partial B_\delta$, $$\max\left\{|u|_{L^{r(\cdot)}(\varrho,\R^N)}, |u|_{L^{t(\cdot)}(b,\R^N)} \right\}\leq C_7\|u\|<C_7\delta_0<1\ \text{ and } \ \|u\|<\delta_0<1$$
	in view of \eqref{4.EstNorms} and \eqref{4.delta0}. Thus,
	by utilizing $(\mathcal{F}_2)$, \eqref{Est.A}, \eqref{Est.hat.M0}, \eqref{4.EstNorms} and invoking Proposition~\ref{norm-modular} we have
\begin{align}\label{P.le4.4-1}
	\notag\widetilde{J}_\lambda(u)&\geq \widetilde{ m}_0\mathcal{A}(u)-C_5\lambda\int_{\R^N}\varrho(x)|u|^{r(x)}\diff x-\left(C_6\lambda+\frac{1}{t^-}\right)\int_{\R^N}b(x)|u|^{t(x)}\diff x\\
	\notag&\geq \widetilde{ m}_0\mathcal{A}(u)-C_5\lambda\int_{\R^N}\varrho(x)|u|^{r(x)}\diff x-\frac{2}{t^-}\int_{\R^N}b(x)|u|^{t(x)}\diff x\\
	&\geq \widetilde{ m}_0\alpha_1\delta^{q_\alpha^+}-C_7^{r^-}\lambda\delta^{r^-}-\frac{2}{t^-}C_7^{t^-}\delta^{t^-},\quad \forall u\in\partial B_\delta.
\end{align}
Noticing $\frac{2}{t^-}C_7^{t^-}\delta^{t^-}\leq \frac{1}{2}\widetilde{ m}_0\alpha_1\delta^{q_\alpha^+}$, inequality \eqref{P.le4.4-1} yields
\begin{align}
	\notag\widetilde{J}_\lambda(u)\geq \frac{1}{2}\widetilde{ m}_0\alpha_1\delta^{q_\alpha^+}-C_7^{r^-}\lambda\delta^{r^-}=C_7^{r^-}\delta^{r^-}\left(\frac{1}{2}\widetilde{m}_0\alpha_1C_7^{-r^-}-\lambda\right), \quad\forall u\in\partial B_\delta.
\end{align}
So, by taking $\lambda_2:=\min\left\{(C_6t^-)^{-1},\frac{1}{2}\widetilde{m}_0\alpha_1C_7^{-r^-}\right\}$ and $\rho_\lambda:=C_7^{r^-}\delta^{r^-}\left(\frac{1}{2}\widetilde{m}_0\alpha_1C_7^{-r^-}-\lambda\right)$, we deduce that for $\lambda\in (0,\lambda_2)$ it holds
\begin{align}
	\notag\widetilde{J}_\lambda(u)\geq \rho_\lambda>0, \quad\forall u\in\partial B_\delta.
\end{align}
The proof is complete.
\end{proof}

Set
\begin{equation}\label{lambda3}
	\lambda_3:=\beta\left(C_2t^+\right)^{-1},
\end{equation}
where $\beta$ and $C_2$ are taken from $(\mathcal{F}_2)$. Then, we have the following.
%----------LEMMA 3.6(NEGATIVE ENERGY WITH |u| SMALL)----------- %
\begin{lemma} \label{Le4.5}
	Let $\lambda\in (0,\lambda_3)$ with $\lambda_3$ given by \eqref{lambda3}. Then, for any $\phi\in X\setminus\{0\}$ with $\phi\geq 0$ we find $\tau_{\lambda,\phi}\in (0,1)$ such that
	$$\widetilde{J}_\lambda(\tau\phi)<0,\quad \forall \tau\in (0,\tau_{\lambda,\phi}).$$
\end{lemma}
\begin{proof}
	Let $\tau\in (0,1)$. By \eqref{Est.A}, \eqref{Est.hat.M0} and $(\mathcal{F}_2)$ we have
	\begin{align*}
		\widetilde{J}_\lambda(\tau\phi)\leq &M(\bar{\tau}_0)\int_{\mathbb{R}^N}A(x,\nabla (\tau\phi)) \diff x + \frac{1}{\alpha^{^-}}\int_{\R^N}V(x)(\tau\phi)^{\alpha(x)} \diff x\\
		&-\frac{\la }{\beta}\int_{\R^N}\left[C_1\varrho(x)(\tau\phi)^{r(x)}-C_2b(x)(\tau\phi)^{t(x)}\right] \diff x-\frac{1}{t^{+}}\int_{\R^N}b(x)(\tau\phi)^{t(x)} \diff x\\
		\leq &(1+M(\bar{\tau}_0))\mathcal{A}(\tau\phi)-\frac{\la C_1 }{\beta}\int_{\R^N}\varrho(x)(\tau\phi)^{r(x)} \diff x-\left(\frac{1}{t^{+}}-\frac{C_2\la}{\beta}\right)\int_{\R^N}b(x)(\tau\phi)^{t(x)} \diff x\\
		\leq &(1+M(\bar{\tau}_0))\left(1+\|\phi\|^{q_\alpha^+}\right)\tau^{p_\alpha^-}-\frac{\la C_1 }{\beta}\left(\int_{\R^N}\varrho(x)\phi^{r(x)} \diff x\right)\tau^{r^+}.
	\end{align*}
	Since $r^+<p_\alpha^-$, the conclusion follows from the last estimate.	
\end{proof}

%=======================PROOF OF MAIN THEOREM 2======================%
\begin{proof}[\textbf{Proof of Theorem~\ref{Theo.Sub1}}]
	Let $\delta_*$ be such that
	\begin{equation}\label{P.Theo4.1-1}
		\begin{cases}
			0<\delta_*<\delta_0,\\
			\alpha_2\delta_*^{p_\alpha^-}<\bar{\tau}_0,
		\end{cases}
	\end{equation}
see the notations $p_\alpha$, $\alpha_2$, $\bar{\tau}_0$ and $\delta_0$ in \eqref{Def.p_alpha}, \eqref{Est.A}, \eqref{4.1-t0} and \eqref{4.delta0}, respectively. Let $\lambda_\star$ be such that
\begin{equation}
\begin{cases}\label{P.Theo4.1-2}
	0<\lambda_\star<\min \left\{\lambda_1,\lambda_2(\delta_*),\lambda_3\right\},\\
		0<K_1
	-K_2\max\left\{\lambda_\star^{\frac{l^{+}}{l^{+}-1}},\lambda_\star^{\frac{l^{-}}{l^{-}-1}}\right\},
\end{cases}
\end{equation}
where $K_1$, $K_2$, $\lambda_1$, $\lambda_2(\delta_*)$, and $\lambda_3$ are given in Lemmas~\ref{Le4.3}, \ref{Le4.4} and \ref{Le4.5}. Let $\lambda\in (0,\lambda_\star)$. Then, by Lemma~\ref{Le4.4} we find $\rho_\lambda>0$ such that
 \begin{equation*}
	\widetilde{J}_\lambda(u)\geq \rho_\lambda,\quad \forall u\in \partial B_{\delta_*}.
\end{equation*}
Define
$$c:=\underset{u\in\overline{B}_{\delta_*}}\inf \widetilde{J}_\lambda(u).$$
Then, from Lemma~\ref{Le4.5} and the definition of $c$, we deduce
	$-\infty<c<0.$
	From the Ekeland variational principle (see, for example, \cite{BN}), for each $0<\epsilon<\underset{u\in\partial B_{\delta_*}}\inf \widetilde{J}_\lambda(u)-c,$ we find $u_{\epsilon}\in \overline{B}_{\delta_*} $ such that
	\begin{equation}\label{5.ekeland}
		\begin{cases}
			\widetilde{J}_\lambda(u_{\epsilon})\leq c+\epsilon,\\
			\widetilde{J}_\lambda(u_{\epsilon})<\widetilde{J}_\lambda(u)+\epsilon\|u-u_{\epsilon}\|, \ \forall u\in \overline{B}_{\delta_*} , u \ne u_{\epsilon}.
		\end{cases}
	\end{equation}
	Thus, $u_{\epsilon}\in B_{\delta_*}$ since $\widetilde{J}_\lambda(u_{\epsilon})\leq c+\epsilon<\underset{u\in\partial B_{\delta_*}}\inf \widetilde{J}_\lambda(u).$ Hence, we deduce from \eqref{5.ekeland} that
	\begin{equation}\label{5.estJ1'}
		\|\widetilde{J}_\lambda'(u_\epsilon)\|_{X^\ast}\leq \epsilon.
	\end{equation}
	From \eqref{5.ekeland} and \eqref{5.estJ1'}, we produce a $(\textup{PS})_c$-sequence $\{u_n\}_{n\in\mathbb{N}}$ for $\widetilde{J}_\lambda$. Notice that by \eqref{P.Theo4.1-2} we have $$c<0<K_1
	-K_2\max\left\{\lambda^{\frac{l^{+}}{l^{+}-1}},\lambda^{\frac{l^{-}}{l^{-}-1}}\right\}.$$
	Thus, up to a subsequence we get that $u_n\to u_\lambda$ in $X$ in view of Lemma~\ref{Le4.3}. So, $\widetilde{J}_\lambda'(u_\lambda)= 0$ and $\widetilde{J}_\lambda(u_\lambda)= c<0$. By \eqref{Est.A} and \eqref{P.Theo4.1-1} and noticing $\|u_\lambda\|\leq \delta_*<1$ we have
	$$\int_{\R^N}A(x,\nabla u_\lambda)\,\diff x\leq \mathcal{A}(u_\lambda)\leq \alpha_2\|u_\lambda\|^{p_\alpha^-}\leq \alpha_2\delta_*^{p_\alpha^-}\leq \bar{\tau}_0.$$
	Thus, $u_\lambda$ is a nonnegative solution to problem~\eqref{Prob}. The proof is complete.
\end{proof}

\medskip

\subsection{\textbf{Existence of infinitely many solutions}}${}$

\smallskip
In this subsection, we will prove Theorem~\ref{Theo.Sub} employing the genus theory for the truncated energy functionals. Our argument follows the proof of \cite[Theorem 2.2]{CH.2021}.

Let $0<\lambda<(C_6t^-)^{-1}=:\lambda^{(1)}$. Then, arguing as that leads to \eqref{P.le4.4-1} we have
\begin{align}\label{NNN1}
	J_\lambda(u)\geq  \widetilde{ m}_0\alpha_1\|u\|^{q_\alpha^+}-\lambda C_7^{r^+}\|u\|^{r^-}-\frac{2}{t^-}C_7^{t^+}\|u\|^{t^-},\quad \|u\|\leq 1.
\end{align}
Thus,
\begin{equation}\label{49}
J_\lambda(u)\ge g_{\la}(\|u\|)\quad\text{for }\|u\|\le1,
\end{equation}
where $g_{\la}\in C([0,\infty))$ is given by
\begin{equation*}
g_{\la}(\tau):=\widetilde{ m}_0\alpha_1\tau^{q_\alpha^+}-\lambda C_7^{r^+}\tau^{r^-}-\frac{2C_7^{t^+}}{t^-}\tau^{t^-},\quad \tau\geq 0.
\end{equation*}
Rewrite $g_{\la}(\tau)=C_7^{r^+}\tau^{r^-}(h(\tau)-\la)$ with
\begin{equation*}
h(\tau):=a_0\tau^{q_\alpha^{+}-r^{-}}-b_0\tau^{t^{-}-r^{-}},
\end{equation*}
where $a_0:=\widetilde{ m}_0\alpha_1C_7^{-r^+}>0$ and $b_{0}:=\frac{2C_7^{t^+-r^+}}{t^-}>0$. Clearly,
\begin{align}\label{50}
\la^{(2)}
&:=\max_{\tau\ge0}h(\tau)=h\left(\left[\frac{(q_\alpha^{+}-r^{-})a_0}{(t^{-}-r^{-})b_0}\right]^{\frac{1}{t^{-}-q_\alpha^{+}}}\right)\notag\\
&=a_{0}^{\frac{t^{-}-r^{-}}{t^{-}-q_\alpha^{+}}}b_{0}^{\frac{r^{-}-t^{+}}{t^{-}-q_\alpha^{+}}}\left(\frac{q_\alpha^{+}-r^{-}}{t^{-}-r^{-}}\right)^{\frac{t^{+}-r^{-}}{q_\alpha^{-}-t^{+}}} \frac{t^{-}-q_\alpha^{+}}{t^{-}-r^{-}}>0
\end{align}
and for any $\la\in\left(0,\la^{(2)}\right)$, $g_{\la}(\tau)$ has only positive roots $\tau_1(\la)$ and $\tau_2(\la)$ with
\begin{equation}\label{P.Theo4.2.tau*}
0<\tau_1(\lambda)<\left[\frac{(q_\alpha^{+}-r^{-})a_0}{(t^{-}-r^{-})b_0}\right]^{\frac{1}{t^{-}-q_\alpha^{+}}}=:\tau_{*}<\tau_2(\lambda).
\end{equation}
Obviously, on $[0,\infty)$ it holds that $g_{\la}(\tau)<0$ if and only if $\tau\in(0,\tau_1(\lambda))\cup(\tau_2(\lambda),\infty)$. Moreover, we have
\begin{equation}\label{51}
\lim_{\la\to0^{+}}\tau_{1}(\la)=0.
\end{equation}
By \eqref{51}, we find $\la_*^{(1)}$ such that
\begin{equation}\label{52}
	\begin{cases}
		0<\la_*^{(1)}<\min\left\{\la_1,\la_3,\la^{(1)},\la^{(2)}\right\},\\
	\tau_1(\lambda)<1,\ \alpha_2\tau_{1}(\lambda)^{p_\alpha^{-}}< \alpha_1\min\left\{1,\tau_*^{q_\alpha^+}\right\},\ \ \forall\la\in(0,\la_*^{(1)}),
	\end{cases}
\end{equation}
where $\la_1$ and $\la_3$ are given in \eqref{lambda1} and \eqref{lambda3}, respectively. For each $\la\in(0,\la_{*}^{(1)})$, we consider the truncated functional $T_{\la}:X\to\mathbb{R}$ given by
\begin{multline*}
T_{\la}(u)
:=\widehat{M}_0\left(\int_{\R^N}A(x,\nabla u)\diff x\right)+\int_{\R^N}\frac{V(x)}{\alpha(x)}|u|^{\alpha(x)}\diff x\\
\quad -\phi\left(\mathcal{A}(u)\right)\left[\la\int_{\R^N}F(x,u)\diff x+\int_{\R^N}\frac{b(x)}{t(x)}|u|^{t(x)}\diff x\right]
\end{multline*}
for $u\in X$, where $\mathcal{A}$ is given by \eqref{Def.A} and $\phi \in C_{c}^{\infty}(\mathbb{R})$, $0\le\phi(\tau)\le1$ for all $\tau\in\mathbb{R}$, $\phi(\tau)=1$ for $|\tau|\le \alpha_2\tau_{1}(\lambda)^{p_\alpha^{-}}$ and $\phi(\tau)=0$ for $|\tau|\ge \alpha_1\min\left\{1,\tau_*^{q_\alpha^+}\right\}$. Clearly, $T_{\la}\in C^{1}(X,\mathbb{R})$ and
\begin{equation}\label{54}
T_{\la}(u)\ge J_\lambda(u)\ \ \text{for all} \ \ u\in X.
\end{equation}
Moreover, we have
\begin{equation}\label{55}
	T_{\la}(u)=J_\lambda(u) \ \ \text{if} \ \ \mathcal{A}(u)<\alpha_2\tau_{1}(\lambda)^{p_\alpha^{-}},
\end{equation}
and
\begin{equation}\label{56}
T_{\la}(u)=\widehat{M}_0\left(\int_{\R^N}A(x,\nabla u)\diff x\right)+\int_{\R^N}\frac{V(x)}{\alpha(x)}|u|^{\alpha(x)}\diff x \ \ \text{if} \ \ \mathcal{A}(u)>\alpha_1\min\left\{1,\tau_*^{q_\alpha^+}\right\}.
\end{equation}

\begin{lemma}\label{Le4.6}
Let $\la\in(0,\la_*^{(1)})$. Then, it holds that
\begin{equation}
\mathcal{A}(u)<\alpha_2\tau_{1}(\lambda)^{p_\alpha^{-}},\ T_{\la}(u)=J_\lambda(u),\ T_{\la}'(u)=J_\lambda'(u)
\end{equation}
for any $u\in X$ satisfying $T_{\la}(u)<0$.
\end{lemma}
\begin{proof}
Let $u\in X$ satisfy  $T_{\la}(u)<0$. Then, we have $J_\lambda(u)<0$ due to \eqref{54}. If $\|u\|>1$, then by \eqref{Est.A}, $\mathcal{A}(u)>\alpha_1$ and therefore, \eqref{56} yields
\begin{equation*}
\widehat{M}_0\left(\int_{\R^N}A(x,\nabla u)\diff x\right)+\int_{\R^N}\frac{V(x)}{\alpha(x)}|u|^{\alpha(x)}\diff x =T_\lambda(u)<0,
\end{equation*}
a contradiction. Thus, we have $\|u\|\le1$ and hence, it follows from \eqref{49} that $g_{\la}(\|u\|)\le J_\lambda(u)<0$. Thus, $\|u\|<\tau_1(\lambda)$ or $\|u\|>\tau_2(\lambda)>\tau_{*}$ due to \eqref{P.Theo4.2.tau*}. If $\|u\|>\tau_{*}$, then by \eqref{Est.A} again we have
$$\mathcal{A}(u)\geq \alpha_1\|u\|^{q_\alpha^{+}}>\alpha_1\tau_{*}^{q_\alpha^{+}};$$
hence, $T_{\la}(u)\ge0$ due to \eqref{56}, a contradiction. Thus, $\|u\|<\tau_1(\lambda)$; hence, $\mathcal{A}(u)<\alpha_2\tau_{1}(\lambda)^{p_\alpha^{-}}$. Then the conclusion follows from \eqref{55} and the definition of $T_\lambda$. The proof is complete.
\end{proof}
Next, we will construct sequence of critical points $\{u_n\}_{n\in\N}$ of $T_{\la}$ such that $\mathcal{A}(u_n)<\alpha_2\tau_{1}(\lambda)^{p_\alpha^{-}}$ via genus theory. Let us denote by $\Sigma$ the set of all closed subset $E\subset X\setminus\{0\}$ such that $E=-E$, namely, $u\in E$ implies $-u\in E$. For $E\in\Sigma$, let us denote the genus of a $E$ by $\gamma(E)$ (see \cite{Rab1992} for the definition and properties of the genus).
\begin{lemma}\label{Le4.7}
	Let $\la\in(0,\la_*^{(1)})$. Then, for each $k\in\mathbb{N}$, there exists $\epsilon>0$ such that
	\begin{equation}
		\gamma(T_{\la}^{-\epsilon})\ge k,
	\end{equation}
	where $T_{\la}^{-\epsilon}:=\{u\in X:T_{\la}(u)\le-\epsilon\}$.
\end{lemma}
From the assumption of $A$ and $f$, $T_{\la}^{-\epsilon}$ is a closed subset of $X\setminus\{0\}$ and is symmetric with respect to the origin; hence, $\gamma(T_{\la}^{-\epsilon})$ is well defined.
\begin{proof}[\textbf{Proof of Lemma~\ref{Le4.7}}]
	Let $k\in\mathbb{N}$ and let $X_{k}$ be a subspace of $X$ of dimension $k$. Since all norms on $X_k$ are mutually equivalent, we find $\delta_{k}>\tau_1(\lambda)^{-1}(>1)$ such that
	\begin{equation}\label{P.Le4.3-1}
		\delta_{k}^{-1}|u|_{L^{r(\cdot)}(\varrho,\R^N)}\le\|u\|\le\delta_{k}|u|_{L^{r(\cdot)}(\varrho,\R^N)}\quad\forall u\in X_{k}.
	\end{equation}
	For $u\in X_{k}$ with $\|u\|<\delta_{k}^{-1}$, thus, $\max\{\|u\|, |u|_{L^{r(\cdot)}(\varrho,\R^N)}\}<1$ and $\mathcal{A}(u)<\alpha_2\tau_1(\lambda)^{p_\alpha^{-}}$ due to \eqref{P.Le4.3-1} and \eqref{Est.A}, by arguing as the proof of Lemma~\ref{Le4.5} we deduce from Proposition~\ref{norm-modular}, ($\mathcal{F}_2$), \eqref{4.EstNorms} and \eqref{55} that
	\begin{equation*}
		T_{\la}(u)=J_\lambda(u)\le (M(\bar{\tau}_0)+1)\|u\|^{p_\alpha^{-}}-\frac{\la C_1}{\beta}|u|_{L^{r(\cdot)}(\varrho,\R^N)}^{r^{+}}\le (M(\bar{\tau}_0)+1)\|u\|^{p_\alpha^{-}}-\frac{\la C_1 C_7^{r^+}}{\beta}\|u\|^{r^{+}}.
	\end{equation*}
	That is,
	\begin{equation*}
		T_{\la}(u)\le(M(\bar{\tau}_0)+1)\|u\|^{r^{+}}\left(\|u\|^{p_\alpha^{-}-r^{+}}-\frac{\la C_1 C_7^{r^+}}{(M(\bar{\tau}_0)+1)\beta}\right).
	\end{equation*}
	Thus, by taking $\eta$ with $$0<\eta<\min\left\{\delta_{k}^{-1},\left[\frac{\la C_1 C_7^{r^+}}{(M(\bar{\tau}_0)+1)\beta}\right]^{\frac{1}{p_\alpha^{-}-r^{+}}}\right\}$$ and $$\epsilon:=-(M(\bar{\tau}_0)+1)\eta^{r^{+}}\left(\eta^{p_\alpha^{-}-r^{+}}-\frac{\la C_1 C_7^{r^+}}{(M(\bar{\tau}_0)+1)\beta}\right)>0,$$ we have $T_{\la}(u)<-\epsilon<0,$
 for all $u\in S_{\eta}:=\{u\in X_{k}: \|u\|=\eta\}$. This yields $S_{\eta}\subset T_{\la}^{-\epsilon}$; hence, $\gamma(T_{\la}^{-\epsilon})\ge\gamma(S_{\eta})=k$
	and this completes the proof.
\end{proof}
Now, for each $k\in\mathbb{N}$, define
\begin{equation*}
	\Sigma_k:=\{E\in\Sigma: \gamma(E)\ge k\}
\end{equation*}
and
\begin{equation*}
	c_k:=\inf_{E\in \Sigma_k}\sup_{u\in E}T_{\la}(u).
\end{equation*}
As in \cite[Lemma 3.9]{Fig-Nas.EJDE2016}, we have
\begin{equation}\label{P.Theo4.2-ck}
	c_k<0,\ \ \forall k\in\N.
\end{equation}
Let $\la_{*}^{(2)}>0$ be such that
\begin{equation*}
	K_1-K_2\max\left\{\left(\la_{*}^{(2)}\right)^{\frac{l^{+}}{l^{+}-1}},\left(\la_{*}^{(2)}\right)^{\frac{l^{-}}{l^{-}-1}}\right\}>0,
\end{equation*}
where $K_1$, $K_2$ and $l$ are as in Lemma~\ref{Le4.3}. Set
\begin{equation}\label{57}
	\la_{*}:=\min\{\la_{*}^{(1)},\la_{*}^{(2)}\},
\end{equation}
where $\la_{*}^{(1)}$ is given by \eqref{52}. We have the following.
\begin{lemma}\label{Le4.8}
	For $\la\in(0,\la_{*})$, if $c=c_{k}=c_{k+1}=\cdots=c_{k+m}$ for some $m\in\mathbb{N}$, then
	\begin{equation*}
		\gamma(K_c)\ge m+1,
	\end{equation*}
	where $K_{c}:=\{u\in X\setminus\{0\}:T_{\la}'(u)=0\text{ and }T_{\la}(u)=c\}$.
\end{lemma}
\begin{proof}
	Let $\la\in(0,\la_{*})$. Then, by \eqref{P.Theo4.2-ck} and the choice of $\la_{*}$ we have
	\begin{equation*}
		c<0<	K_1-K_2\max\left\{\lambda^{\frac{l^{+}}{l^{+}-1}},\lambda^{\frac{l^{-}}{l^{-}-1}}\right\}.
	\end{equation*}
	Thus, $K_c$ is a compact set in view of Lemmas \ref{Le4.3} and \ref{Le4.6}. Using this fact and a standard argument for which the deformation lemma is applied, we derive the desired conclusion (see, for example, \cite[Lemma 4.4]{Ber-Gar-Per.1996} or \cite[Lemma 3.10]{Fig-Nas.EJDE2016}.
\end{proof}
\begin{proof}[\textbf{Proof of Theorem~\ref{Theo.Sub}}]
	Let $\la_{*}$ be defined as in \eqref{57}. Let $\la\in(0,\la_{*})$. In view of Lemma \ref{Le4.8}, we find sequence $\{u_n\}_{n\in\mathbb{N}}$ of critical points of  $T_{\la}$  with $T_{\la}(u_n)<0$ for all $n\in\mathbb{N}$. By Lemma~\ref{Le4.6}, $\{u_n\}_{n\in\mathbb{N}}$ are solutions to problem~\eqref{Prob}. Let us denote by $u_{\la}$ one of $u_n$. By Lemma~\ref{Le4.6} again, we have
	\begin{equation*}
		\mathcal{A}(u_n)<\alpha_2\tau_{1}(\lambda)^{p_\alpha^{-}}.
	\end{equation*}
	From this and \eqref{Est.A} we arrive at
	\begin{equation*}
		\lim_{\la\to 0^{+}}\|u_{\la}\|=0,
	\end{equation*}
	and the proof is complete.
\end{proof}

\appendix
\section{Proof of Inequality \eqref{Est.A}}

\begin{proof}[\textbf{Proof of Inequality \eqref{Est.A}}] We will prove that \eqref{Est.A} holds true with
	 $$\alpha_1=
	\min\left\{M_{1},\frac{1}{\alpha^{+}}\right\}3^{-q_\alpha^+}\ \ \text{and} \ \ \alpha_2=3\max\left\{2^{2q^+}M_{2},\frac{1}{\alpha^{-}}\right\}.$$
	Let $u\in X$ and we first prove the lower bound of $\mathcal{A}(u)$ as follows. By \eqref{Est1} we have
	\begin{align}\label{p1}
		\notag\mathcal{A}(u)&\ge M_{1}\left[\int_{\R^{N}}|\chi_{\{|\nabla u|\geq 1\}}\nabla u|^{p(x)}\diff x+\int_{\R^{N}}|\chi_{\{|\nabla u|\leq 1\}}\nabla u|^{q(x)}\diff x\right]+\frac{1}{\alpha^{+}}\int_{\R^{N}}V(x)|u|^{\alpha(x)}\diff x\\
		&\geq m_0\mathcal{B}(u),
			\end{align}
		where
		$$\mathcal{B}(u):=\int_{\R^{N}}\left[|\chi_{\{|\nabla u|\geq 1\}}\nabla u|^{p(x)}+|\chi_{\{|\nabla u|\leq 1\}}\nabla u|^{q(x)}+V(x)|u|^{\alpha(x)}\right]\diff x\ \  \text{and} \ \ m_0:=\min\left\{M_{1},\frac{1}{\alpha^{+}}\right\}.$$
	For brevity, set
	$$a_0:=\big||\chi_{\{|\nabla u|\geq 1\}}\nabla u|\big|_{L^{p(\cdot)}(\R^{N})},\ b_0:=\big||\chi_{\{|\nabla u|\leq 1\}}\nabla u|\big|_{L^{q(\cdot)}(\R^{N})},\ \text{and}\  c_0:=|u|_{L^{\alpha(\cdot)}(V,\R^{N})}.$$ Thus, it holds $$a_0+b_0+c_0\ge\|u\|.$$
	\vskip5pt
	$\bullet$ \textit{Case $\mathcal{A}(u)<m_{0}.$}
	
	\vskip5pt
	In this case, we get from \eqref{p1} that
		\begin{align*}
		\int_{\R^{N}}|\chi_{\{|\nabla u|\geq 1\}}\nabla u|^{p(x)}\diff x+\int_{\R^{N}}|\chi_{\{|\nabla u|\leq 1\}}\nabla u|^{q(x)}\diff x+\int_{\R^{N}}V(x)|u|^{\alpha(x)}\diff x<1.
	\end{align*}
	Consequently, by invoking Proposition~\ref{norm-modular} we derive from \eqref{p1} that
	\begin{align*}%\label{p2}
		\mathcal{A}(u)\ge m_{0}\left(a_0^{q_\alpha^+}+b_0^{q_\alpha^+}+c_0^{q_\alpha^+}\right)\
		\ge m_{0}3^{1-q_\alpha^+}\left(a_0+b_0+c_0\right)^{q_\alpha^+}\ge m_{0}3^{1-q_\alpha^+}\|u\|^{q_\alpha^+}.
	\end{align*}
	
	\vskip5pt
	$\bullet$ \textit{Case $\mathcal{A}(u)\ge m_{0}$}.
	
	\vskip5pt
	\begin{itemize}
		\item[-] If $\|u\|\le3$, then
		$$\mathcal{A}(u)\ge m_{0}3^{-q_\alpha^+}\|u\|^{q_\alpha^+}.$$
		
		\vskip5pt
		\item [-] If $\|u\|\ge3$, then $a_0+b_0+c_0\ge3.$ Thus, there exists $x\in \{a_0,b_0,c_0\}$ such that $x\geq 1$.
		
		\vskip5pt
		If $x$ is the only element of $\{a_0,b_0,c_0\}$ that is greater than or equal 1 and the remaining two elements are denoted by $y,z$, then $y,z\le1$ and it holds
		$$
		x\ge\frac{1}{3}(x+y+z)=\frac{1}{3}(a_0+b_0+c_0)\ge1.
		$$
		Thus, by \eqref{p1} and Proposition~\ref{norm-modular} we obtain
		\begin{align*}%\label{p4}
			\mathcal{A}(u)\ge m_0x^{p_\alpha^{-}}\ge m_0\left[\frac{1}{3}(a_0+b_0+c_0)\right]^{p_\alpha^{-}}\ge m_03^{-p_\alpha^{-}}\|u\|^{p_\alpha^{-}}.
		\end{align*}
	
	If there are only two of $\{a_0,b_0,c_0\}$ greater or equal 1, denoted by $x,y$, and the remaining element is denoted by $z$, then $x,y\ge 1> z$ and
	$$
	x+y\ge\frac{2}{3}(x+y+z)= \frac{2}{3}(a_0+b_0+c_0).
	$$
	Thus, by \eqref{p1} and Proposition~\ref{norm-modular} again we obtain
	\begin{align*}
		\mathcal{A}(u)&\ge m_0\left(x^{p_\alpha^{-}}+y^{p_\alpha^{-}}\right)\ge m_02^{1-p_\alpha^{-}}(x+y)^{p_\alpha^{-}}\\
&\ge m_02^{1-p_\alpha^{-}}\left[\frac{2}{3}(a_0+b_0+c_0)\right]^{p_\alpha^{-}}\geq m_{0}3^{-p_\alpha^{-}}\|u\|^{p_\alpha^{-}}.
	\end{align*}
	
	Finally, for $a_0\ge1$, $b_0\ge1$ and $c_0\ge1$, we have
	$$
	\mathcal{A}(u)\ge m_{0}(a_0^{p_\alpha^{-}}+b_0^{p_\alpha^{-}}+c_0^{p_\alpha^{-}})\geq m_{0}3^{1-p_\alpha^{-}}(a_0+b_0+c_0)^{p_\alpha^{-}}\ge m_{0}3^{1-p_\alpha^{-}}\|u\|^{p_\alpha^{-}}.
	$$
	\end{itemize}
 In any case, we obtain the following lower bound of $\mathcal{A}(u)$:
 $$
 \mathcal{A}(u)\ge m_03^{-q_\alpha^+}\min\left\{\|u\|^{p_\alpha^{-}},\|u\|^{q_\alpha^{+}}\right\}.
 $$
	
	\vskip5pt
	Next, we prove the upper bound of $\mathcal{A}(u)$. To this end, let $v\in(L^{p(\cdot)}(\R^{N}))^{N}$ and $w\in(L^{q(\cdot)}(\R^{N}))^{N}$ be such that
	$$
	\nabla u=v+w\quad\text{and}\quad\big||\nabla u|\big|_{L^{p(\cdot)}(\R^{N})+L^{q(\cdot)}(\R^{N})}=|v|_{L^{p(\cdot)}(\R^{N})}+|w|_{L^{q(\cdot)}(\R^{N})}.
	$$
From this and \eqref{Est1} we obtain
	\begin{align}\label{p6*}
		\mathcal{A}(u)&\le M_{2}\int_{\{|\nabla u|\geq 1\}}|\nabla u|^{p(x)}\diff x+M_{2}\int_{\{|\nabla u|< 1\}}|\nabla u|^{q(x)}\diff x+\frac{1}{\alpha^{-}}\int_{\R^{N}}V(x)|u|^{\alpha(x)}\diff x\notag\\
		&\le M_{2}2^{p^{+}-1}\int_{\{|\nabla u|\geq 1\}}\left[|v|^{p(x)}+|w|^{p(x)}\right]\diff x+M_{2}2^{q^+-1}\int_{\{|\nabla u|< 1\}}\left[|v|^{q(x)}+|w|^{q(x)}\right]\diff x\notag\\
		&\quad+\frac{1}{\alpha^{-}}\int_{\R^{N}}V(x)|u|^{\alpha(x)}\diff x.
	\end{align}
	
	\vskip5pt
	$\bullet$ On $\{|\nabla u|\geq 1\}$, we have
	$$
	|v|+|w|\ge|\nabla u|\ge1.
	$$
	\begin{itemize}
		\item[-] If $|w|\ge\frac{1}{2}$, then we have $|w|^{p(x)}\le2^{(q-p)^{+}}|w|^{q(x)}$, and thus,
		$$
		|v|^{p(x)}+|w|^{p(x)}\le2^{(q-p)^{+}}\left(|v|^{p(x)}+|w|^{q(x)}\right).
		$$
		\item[-] If $|w|<\frac{1}{2}$, then we have $|v|>\frac{1}{2}$, and thus,
		$$
		|v|^{p(x)}+|w|^{p(x)}\le2|v|^{p(x)}\le2\left(|v|^{p(x)}+|w|^{q(x)}\right).
		$$
	\end{itemize}
	
	So, we infer
	\begin{align}\label{p7}
		\int_{\{|\nabla u|\geq 1\}}\left[|v|^{p(x)}+|w|^{q(x)}\right]\diff x\le2^{1+(q-p)^{+}}\int_{\{|\nabla u|\geq 1\}}\left[|v|^{p(x)}+|w|^{q(x)}\right]\diff x.
	\end{align}
	
	\vskip5pt
	$\bullet$ On $\{|\nabla u|< 1\}$, we have $|v+w|=|\nabla u|<1$.
	\begin{itemize}
		\item[-] If $|w|\ge1$, then
		$$
		|v|\le|v+w|+|w|\le2|w|.
		$$
		Thus, we know
		$$
		|v|^{q(x)}+|w|^{q(x)}\le2^{1+q^+}|w|^{q(x)}\le2^{1+q^+}(|v|^{p(x)}+|w|^{q(x)}).
		$$
		\item[-] 	If $|w|<1$, then
		$$
		|v|\le|v+w|+|w|<2.
		$$
		Thus, one has $|v|^{q(x)}\le2^{(q-p)^{+}}|v|^{p(x)}.$
		Hence, we get $$|v|^{q(x)}+|w|^{q(x)}\le2^{(q-p)^{+}}(|v|^{p(x)}+|w|^{q(x)}).$$
	\end{itemize}

	Thus, it holds
	\begin{align}\label{p8}
		\int_{\{|\nabla u|< 1\}}\left[|v|^{q(x)}+|w|^{q(x)}\right]\diff x\le2^{1+q^+}\int_{\{|\nabla u|< 1\}}\left[|v|^{p(x)}+|w|^{q(x)}\right]\diff x.
	\end{align}
	From \eqref{p6*}-\eqref{p8}, we obtain
	\begin{align*}
		\mathcal{A}(u)&\le M_{2}2^{p^{+}+(q-p)^{+}}\int_{\{|\nabla u|\geq 1\}}\left[|v|^{p(x)}+|w|^{q(x)}\right]\diff x+M_{2}2^{2q^+}\int_{\{|\nabla u|< 1\}}\left[|v|^{p(x)}+|w|^{q(x)}\right]\diff x\notag\\
		&\quad+\frac{1}{\alpha^{-}}\int_{\R^{N}}V(x)|u|^{\alpha(x)}\diff x\\
		&\le2^{2q^+}M_{2}\left(\int_{\R^{N}}|v|^{p(x)}\diff x+\int_{\R^{N}}|w|^{q(x)}\diff x\right)+\frac{1}{\alpha^{-}}\int_{\R^{N}}V(x)|u|^{\alpha(x)}\diff x.
	\end{align*}
	Applying Proposition~\ref{norm-modular} again we drive from the last inequality that
	\begin{align}\label{p9}
		\mathcal{A}(u)&\le2^{2q^+}M_{2}\left[\max\left\{|v|_{L^{p(\cdot)}(\R^{N})}^{p^{-}},|v|_{L^{p(\cdot)}(\R^{N})}^{p^{+}}\right\} +\max\left\{|w|_{L^{q(\cdot)}(\R^{N})}^{q^{-}},|w|_{L^{q(\cdot)}(\R^{N})}^{q^+}\right\}\right]\notag\\
		&\quad+\frac{1}{\alpha^{-}}\max\left\{|u|_{L^{\alpha(\cdot)}(V,\R^{N})}^{\alpha^{-}},|u|_{L^{\alpha(\cdot)}(V,\R^{N})}^{\alpha^{+}}\right\}.
	\end{align}
	Note that $|v|_{L^{p(\cdot)}(\R^{N})}+|w|_{L^{q(\cdot)}(\R^{N})}+|u|_{L^{\alpha(\cdot)}(V,\R^{N})}=\|u\|$. We will distinguish two cases.
	\begin{itemize}
		\item [-] If $\|u\|\ge1$, then by \eqref{p9}, we have
		\begin{align*}%\label{p10}
			\mathcal{A}(u)&\le2^{2q^+}M_{2}\left(\|u\|^{p^{+}}+\|u\|^{q^+}\right)+\frac{1}{\alpha^{-}}\|u\|^{\alpha^{+}}\notag\\
			&\le3\max\left\{2^{2q^+}M_{2},\frac{1}{\alpha^{-}}\right\}\|u\|^{q_\alpha^+}.
		\end{align*}
	\item [-] If $\|u\|<1$, then $|v|_{L^{p(\cdot)}(\R^{N})}<1$, $|w|_{L^{q(\cdot)}(\R^{N})}<1$ and $|u|_{L^{\alpha(\cdot)}(V,\R^{N})}<1$, and thus, \eqref{p9} yields
	\begin{align*}
		\mathcal{A}(u)&\le2^{2q^+}M_{2}\left(|v|_{L^{p(\cdot)}(\R^{N})}^{p^{-}}+|w|_{L^{q(\cdot)}(\R^{N})}^{q^{-}}\right)+\frac{1}{\alpha^{-}}|u|_{L^{\alpha(\cdot)}(V,\R^{N})}^{\alpha^{-}}\\
		&\le\max\left\{2^{2q^+}M_{2},\frac{1}{p_\alpha^{-}}\right\}\left(|v|_{L^{p(\cdot)}(\R^{N})}^{p_\alpha^{-}}+|w|_{L^{q(\cdot)}(\R^{N})}^{p_\alpha^{-}}+|u|_{L^{\alpha(\cdot)}(V,\R^{N})}^{p_\alpha^{-}}\right)\\
		&\le\max\left\{2^{2q^+}M_{2},\frac{1}{\alpha^{-}}\right\}\left(|v|_{L^{p(\cdot)}(\R^{N})}+|w|_{L^{q(\cdot)}(\R^{N})}+|u|_{L^{\alpha(\cdot)}(V,\R^{N})}\right)^{p_\alpha^{-}}=\max\left\{2^{2q^+}M_{2},\frac{1}{\alpha^{-}}\right\}\|u\|^{p_\alpha^{-}}.
	\end{align*}
		\end{itemize}
	
	In summary, we obtain the upper bound of $\mathcal{A}(u)$ as
	$$
	\mathcal{A}(u)\le3\max\left\{2^{2q^+}M_{2},\frac{1}{\alpha^{-}}\right\}\max\left\{\|u\|^{p_\alpha^{-}},\|u\|^{q_\alpha^+}\right\}.
	$$
	The proof is complete.
\end{proof}

%$$$$$$$$$$$$$$$$$$$$$$$$$$$$$ REFERENCES $$$$$$$$$$$$$$$$$$$$$$$$$$$$$%
\subsection*{Acknowledgements}
K. Ho was partially supported by the National Research Foundation of Korea (NRF) grant funded by the Korea government (MSIT) (grant No. 2022R1A4A1032094). Y.-H. Kim was supported by the Basic Science Research Program through the National Research Foundation of Korea (NRF) funded by the Ministry of Education (NRF-2019R1F1A1057775). C. Zhang was supported by the National Natural Science Foundation of China (No. 12071098).

\end{document}